\documentclass[11 pt]{amsart}

\usepackage{amsmath}
\usepackage{amssymb}
\usepackage{amsfonts}
\usepackage{amsthm}
\usepackage{enumerate}
\usepackage[mathscr]{eucal}
\usepackage[all]{xy}
\usepackage{graphicx}
\usepackage{colortbl}

\usepackage{tikz}
\usepackage{bbm}

\usepackage{verbatim}

\newcommand{\ci}[1]{_{{}_{\!\scriptstyle{#1}}}}

\newtheorem{thm}{Theorem}[section]
\newtheorem{lemma}[thm]{Lemma}
\newtheorem{corollary}[thm]{Corollary}
\newtheorem{proposition}[thm]{Proposition}
\newtheorem{definition}[thm]{Definition}

\numberwithin{equation}{section}

\newcommand{\Be}{\begin{equation}}
\newcommand{\Ee}{\end{equation}}
\newcommand{\Bea}{\begin{align}}
\newcommand{\Eea}{\end{align}}
\newcommand{\Beas}{\begin{align*}}
\newcommand{\Eeas}{\end{endalign*}}
\newcommand{\Benu}{\begin{enumerate}}
\newcommand{\Eenu}{\end{enumerate}}
\newcommand{\Bi}{\begin{itemize}}
\newcommand{\Ei}{\end{itemize}}

\newcommand\bbone{{\mathbbm 1}}

\def\vs{{\varsigma}}
\def\oxi{{\overline \xi}}
\def\ux{{\underline x}}
\def\uy{{\underline y}}
\def\uX{{\underline X}}
\def\uY{{\underline Y}}
\def\ua{{\underline a}}
\def\ub{{\underline b}}
\def\uw{{\underline w}}
\def\dil{{\text{\rm Dil}}}

\def\intslash{\rlap{\kern  .32em $\mspace {.5mu}\backslash$ }\int}
\def\qsl{{\rlap{\kern  .32em $\mspace {.5mu}\backslash$ }\int_{Q_x}}}

\def\emph#1{{\it #1 }}

\def\Ga{\Gamma}
\def\ga{\gamma}

\def\cf{{\it cf}}

\def\supp{{\text{\rm supp }}}

\def\inn#1#2{\langle#1,#2\rangle}

\def\noi{\noindent}

\def\meas{{\text{\rm meas}}}

\def\lc{\lesssim}
\def\gc{\gtrsim}

\def\eps{\varepsilon}

\def\ka{\kappa}
            
\def\la{\lambda}

\def\vphi{\varphi}
             
              \def\Om{\Omega}

\def\bbH{{\mathbb {H}}}

\def\bbN{{\mathbb {N}}}

\def\bbR{{\mathbb {R}}}

\def\bbV{{\mathbb {V}}}

\def\bbZ{{\mathbb {Z}}}

\def\cK{{\mathcal {K}}}

\def\cM{{\mathcal {M}}}

\def\cQ{{\mathcal {Q}}}
\def\cR{{\mathcal {R}}}

\def\cT{{\mathcal {T}}}
\def\cU{{\mathcal {U}}}
\def\cV{{\mathcal {V}}}

\def\cZ{{\mathcal {Z}}}

\def\tx{{\tilde x}}
\def\tu{{\tilde u}}

\def\be#1{\begin{equation}\label{#1}}
\def\endeq{\end{equation}}
\def\endal{\end{align}}
\def\bas{\begin{align*}}
\def\eas{\end{align*}}
\def\bi{\begin{itemize}}
\def\ei{\end{itemize}}

\begin{document}
\title[Stability of  a maximal function estimate]
{Spherical means on the Heisenberg group:\\  Stability of  a maximal function estimate}
\author
[T. Anderson \ \ L. Cladek \ \ M. Pramanik \ \ A. Seeger]
{Theresa C. Anderson \  \ Laura Cladek \\  Malabika Pramanik \  \ Andreas Seeger}
\subjclass[2010]{42B25, 22E25, 43A80, 35S30}
\thanks{Preliminary version. Research supported in part by the National Science Foundation}
%\date{Preliminary version, December 31, 2017.}

\address{Theresa C. Anderson \\ Department of Mathematics \\ University of Wisconsin \\480 Lincoln Drive\\ Madison, WI
53706, USA} \email{tcanderson@math.wisc.edu}

\address{Laura Cladek\\ Department of Mathematics\\ UCLA \\
Box 951555\\ 
Los Angeles, CA 90095-1555, USA}
\email{cladekl@math.ucla.edu}

\address{Malabika Pramanik\\ Department of Mathematics
\\
University of British Columbia\\
1984 Mathematics Road\\ 
Vancouver, B.C. \\Canada V6T 1Z2 }
 \email{malabika@math.ubc.ca}	
\address{Andreas Seeger \\ Department of Mathematics \\ University of Wisconsin \\480 Lincoln Drive\\ Madison, WI
53706, USA} \email{seeger@math.wisc.edu}
\maketitle 

\begin{abstract} Consider the surface measure $\mu$ on a sphere in a nonvertical hyperplane on the Heisenberg group $\mathbb{H}^n$, $n\ge 2$, and the convolution $f*\mu$.  Form the associated maximal function $Mf=\sup_{t>0}|f*\mu_t|$  generated by the automorphic dilations. 
We use decoupling inequalities due to Wolff and Bourgain-Demeter
 to prove   $L^p$-boundedness of $M$ in an optimal range.
 \end{abstract}
\section{Introduction}

Let $\bbH^n$ be the Heisenberg group of Euclidean dimension $2n+1$, with the group law
$$(x,u)\cdot (y,v)=(x+y, u+v+ x^\intercal Jy) \,$$
%\text{  where $J=\begin{pmatrix} 0&-I_n\\I_n&0\end{pmatrix}$ }$$
with  $J$ is a nondegenerate skew symmetric $2n\times 2n$ matrix.
%and $\beta\neq 0$. 
Consider $\bbH^n$ as the vector space $ \bbR^{2n+1}$ 
%$ as \bbH^n$  as a vector space
and let 
 $\bbV$ be a linear subspace of dimension $2n$  which does not contain the center of $\bbH^n$, i.e.  $$\bbV\equiv \bbV_\la=\{(x,\la(x))\}$$
is the graph of a linear functional  $\la:\bbR^{2n}\to \bbR$. Let $\Sigma$ be a convex hypersurface in $\bbV$, with {\it nonvanishing curvature},  which contains the origin in its interior. Note that  $\Sigma $ is a surface of codimension two in $\bbH^n$. Let $\mu$ be a smooth density on $\Sigma$, that is  $\mu=\chi d\sigma$ where $d\sigma$ is  surface measure on $\Sigma$ and $\chi\in C^\infty_c$.

The  natural dilation group $\{\dil_t\}_{t>0}$ of automorphisms on $\bbH^n$ is given by 
$$(x,u)\mapsto \dil_t (x,u)=(tx, t^2 u)$$
where $x\in \bbR^{2n}$, $u\in \bbR$. Define
the dilated measure $\mu_t\equiv \dil_t\mu$ by its action on
Schwartz functions $f$, 
\Be
\inn {\mu_t}{f} =\int f(\dil_t(x,u))d\mu(x,u)
\Ee
and consider the maximal function 
\Be \label{maxop}\cM f= \sup_{t>0} |f*\mu_t| \Ee
where the convolution refers to the noncommutative convolution on the Heisenberg group
(see \eqref{Hconvol}).
%That is, we have, for $f$ a Schwartz function,
%\[f*\mu(x,u)= \int\]

The purpose of this paper is to prove a sharp result on $L^p$ boundedness for $n\ge 2$;
a corresponding estimate on $\bbH^1$ will remain open. The problem for $n\ge 2$ was first taken up  in a paper by Nevo and Thangavelu  \cite{NevoTh} who considered the spherical measure on
$\bbR^{2n}\times \{0\}$ (i.e. the case 
$\la=0$) and proved $L^p$ boundedness for the maximal operator 
 in the non-optimal range $p>\frac{2n-1}{2n-2}$. 
 An optimal  result for $\la=0$  was proved by D. M\"uller and one of the authors in \cite{MuSe}. There it is shown that $L^p$ boundedness holds for  $p>\frac{2n}{2n-1}$ when $n\ge 2$ and $\la=0$.  In the case $\la\neq 0$ the paper \cite{MuSe} only has a   non-optimal 
 result, proving $L^p$ boundedness of the maximal operator for $p> \frac{ 2n-1/3}{2n-4/3}$. It was also  conjectured that 
 boundedness for $\la\neq 0$ remains true in the larger range $p>\frac{2n}{2n-1}$. 
   The results in   \cite{MuSe} actually cover  the  larger class of 
   {\it M\'etivier groups}  which strictly  contains the class of groups of Heisenberg type (with possibly higher dimensional center).
 We note that for the case $\bbH^n$, 
  $n\ge 2$, $\la=0$ an alternative proof of the result in \cite{MuSe} was   given 
  %later but independently 
   by Narayanan and Thangavelu \cite{NaTh}, who used spectral  theoretic arguments and the representation theory of the Heisenberg group.  
   
   % the problem of bounding the circular maximal operator on $L^p(\bbH^1)$ is presently  open for all $p<\infty$.

The crucial difference between the two cases $\la=0$ and $\la\neq 0$ is that the automorphic 
dilations
$\{\dil_t\}$ act on $V_0$ but not on $V_\la$ for $\la\neq 0$. 
%The invariance of  $V_0$ leads to improved $L^2$ estimates for the time derivatives  $\frac{d}{dt} f*\mu_t$, and then also to efficient $L^2$ estimates for  $(\frac{d}{dt})^{1/2+\eps} f*\mu_t$. 
We refer to \cite{MuSe} for an explanation of this phenomenon in terms of the geometry of the underlying Fourier integral operators with folding canonical relations. For the case $\la\neq 0$ the $L^2$ methods of  both \cite{MuSe} 
and  \cite{NaTh}   are no longer applicable to obtain the optimal range of $L^p$-boundedness.
Here we  use different $L^p$ methods  based on  Wolff's decoupling inequality \cite{Wolff1}, \cite{LaWo} and its recent improvements by Bourgain and Demeter \cite{bourgain-demeter} to  prove the conjecture in \cite{MuSe}  for the Heisenberg groups $\bbH^n$, $n\ge 2$,  for all subspaces $V_\la$.
The  approach is motivated by  previous  results on $L^q$-Sobolev estimates for averaging operators associated to families of curves in \cite{prs}, \cite{pr-se}. In an early   version  for genera\-lized Radon transforms associated with families of curves  in three dimensions (\cite{pr-se-var})   the relevant decoupling inequalities  are  proved by an induction procedure where a scaled version of the constant coefficient decoupling inequality  is combined with  a nonlinear change of variables, at  every stage in the iteration.  We use this idea here as well.
The resulting theorem can be interpreted as a stability result for the maximal function estimate in \cite{MuSe}.
\begin{thm}\label{mainthm}
Let $n\ge 2$, $\Sigma\subset V_\la$  as above  and $\mu$ be a smooth density  on $\Sigma$. Let 
  $M$ as in \eqref{maxop} and  $p>\frac{2n}{2n-1}$. Then $M$ extends to a bounded operator on $L^p(\bbH^n)$.
\end{thm}

{\it Remark:} It is  instructive to note that  an analogous  stability  result may fail for other more regular measures.
For example,  if we let $\nu_\la$ be the measure on $V_\la$ given by $\bbone_B dS$ where $dS$ is the 
$2n$-dimensional Lebesgue measure on $\bbV_\la$ and  $\bbone_B$ is  the characteristic function of the unit ball centered at the origin.
Then results on  maximal and singular Radon transforms \cite{PS} show that for $\la=0$ the maximal operator
$f\mapsto \sup_t |f*\dil_t\nu_0|$ is bounded on $L^p(\bbH^n)$ for $1<p<\infty$. 
However for $\la\neq 0$ the maximal operator 
$f\mapsto \sup_t |f*\dil_t\nu_\la |$ is bounded on $L^p(\bbH^n)$  only for $\frac{2n+1}{2n}<p<\infty$. The local analogue of the latter maximal operator (with dilations parameters in $[1,2]$) shares some properties  with the spherical maximal operator 
(\cf. \cite{SS}), due to the rotation effect of the  nonisotropic dilation structure. One has an example that shows unboundedness that is similar to the example that shows unboundedness of the spherical maximal operator 
$L^p(\bbR^d)$, when $p\le \frac{d}{d-1}$ (cf. \cite{Steinbook3}, \cite{SS}).

\medskip

 {\it This paper.}  In \S\ref{outline} 
   consider regularizations of the measure defined by  dyadic 
 frequency decompositions 
and prove a crucial  $L^p$-Sobolev inequality 
%$L^q\to L^q_{d-1}/q$-Sobolev result for  $q> \frac{2d}{d-1}$, $d=2n+1$  
for  
the convolution $f*\mu$  when acting on  $L^p$ functions with compact support. As a consequence we obtain an estimate for  a restricted version of the  maximal operator where the dilation parameter is taken in a compact subinterval of $\bbR^+$.  In \S\ref{decouplingstepsect} 
 we describe the basic decoupling step.
In \S\ref{iterationsect} 
an iteration and combination with known $L^2$ estimates leads to the proof of  the main Proposition \ref{dyadicq-Sob}. 
In \S\ref{globalsect} we use Calder\'on-Zygmund type arguments to extend this result to obtain Theorem \ref{mainthm}. The appendix contains a basic integration by parts lemma which is useful in checking the details of the decoupling step in \S\ref{decouplingstepsect}.
 
 \section{Main estimates}\label{outline}
 The convolution 
 $f_1*f_2(x,u)=\int f_1(y,v) f_2((y,v)^{-1}(x,u)) dydv$  on the Heisenberg group can be written as 
 %The convolution on the Heisenberg group is  given by
 \Be\label{Hconvol}\begin{aligned} f_1*f_2(x,u)&= \int f_1(y,v) f_2(x-y, u-v+x^\intercal J y) dy dv
 \\
 &= \int f_2(y,v) f_1(x-y, u-v-x^\intercal Jy) dydv
 \end{aligned}\Ee
 Here $J$ is a nondegenerate skew symmetric matrix on $\bbR^{2n}$.
 
 Split $x=(\ux, x_{2n})$ where $\ux\in \bbR^{2n-1}$. We consider a localization of the measure to a graph 
 $x_{2n}=g(\ux)$ on $V_\la$,  where the Hessian of $g$ is nondegenerate. We will  use permutation of variables to reduce to this situation (\cf. the remark in \S\ref{iterationsect}).

 % We split variables $x=(\ux,x_{2n})$, $y=(\uy,y_{2n})$ and let \footnote{This is not really necessary, but allows us to conveniently reference things in \cite{MuSe}}.
%\footnote{\color{red}Make a comment as in \cite{MuSe} - the argument is like in section 3.2. Refer to a later section, maybe 4.\color{black}}
 
The localized  measure $\mu$ can be represented as an oscillatory integral distribution by
%as an oscillatory integral distribution by
\Be\label{measosc}\eta(x,u)
 \iint
e^{i(\sigma(x_{2n}-g(\ux))+\tau (u-\la(x)))}
 d\sigma d\tau
\Ee
where $\eta $ is a smooth compactly supported  function.

Let $\vs_0\in C^\infty_0(\bbR)$ be an even smooth function  such  that
$\vs_0(s)=1$ if $|s|\le 1$ and such that the support of $\vs_0$ is contained in $(-2,2)$.
Let  $\vs_1(s)=\vs_0(s/2)-\vs_0(s)$ and let, for
$k\ge 1$,  $\zeta_k(s)= \vs_1(2^{1-k} s) $. Also let $\zeta_0=\vs_0$.  Then for $k\ge 1$,  
$\zeta_k$ is supported in $\{s: 2^{k-1}\le |s|\le 2^{k+1}\}$, $k\ge 1$, and we have 
$\sum_{k=0} \zeta_k=1$.  

Let 
\Be \label{mukla}
\mu _k(x,u)= \eta(x,u) \iint e^{ i (\sigma(x_{2n}-g(\ux)) +\tau (u-\la(x)))} \zeta_k  ( (\sigma^2+\tau^2)^{1/2})
\, d\sigma\, d\tau
\Ee
 and we decompose \eqref{measosc} as $\sum_{k=0}^\infty \mu_k$ in the sense of distributions.
 
 The maximal function $\sup_{t>0} |f*\dil_t\mu_k|$ is dominated by $C(k)$ times the analogue of the Hardy-Littlewood maximal function on the Heisenberg group.
 Therefore it suffices to consider the case $k\gg1$.
 
 Our main proposition will be

 \begin{proposition} \label{dyadicq-Sob} 
 (i) For $q> \frac{4n}{2n-1}$, 
 $$\|f*\mu_k\|_{L^q(\bbH^n)}  \lc 2^{-k\frac{2n-1}q} \|f\|_{L^q(\bbH^n)}\,.$$
 
 (ii) For $p<\frac{4n}{2n+1} $
$$\|f*\mu_k\|_{L^p(\bbH^n)}  \lc 2^{-k(2n-1)(1-\frac 1p)} \|f\|_{L^p(\bbH^n)}\,.$$
 Moreover,  
 for $1\le s\le 2$ 
 $$\Big\|\frac{d}{ds} f*\dil_s\mu_k\Big\|_{L^p(\bbH^n)}  \lc 2^k 2^{-k(2n-1) (1- \frac 1p)} \|f\|_{L^p(\bbH^n)}\,.$$
 
 The implicit constants are uniform if $\la$ is taken from a compact subset of $(\bbR^{2n})^*$.
  \end{proposition}

 A well known Sobolev imbedding argument gives a sharp bound for the restricted maximal function:
  \begin{corollary} \label{maxdyadic}  For $ p<\frac{4n}{2n+1}$,
 
 $$\big \|\sup_{1/2<s<2} |f* \dil_s \mu_k | \big\|_{L^p(\bbH^n)} \lc 2^{k (\frac{2n}p-2n+1)} \|f\|_{L^p(\bbH^n)}.$$
 \end{corollary}

We use a further decomposition in the $\sigma$-variables, as in \cite{MuSe}.
Let  \begin{align*}
\zeta_{k,0}(\sigma,\tau)&=\vs_1(2^{-k}\sqrt{\sigma^2+|\tau|^2})(1-\vs_0(2^{-k}\sigma))\\
\zeta_{k,l}(\sigma,\tau)&=\vs_1(2^{-k}\sqrt{\sigma^2+|\tau|^2})\vs_1(2^{l-k}\sigma)
\\
\widetilde \zeta_{k}(\sigma,\tau)&=\vs_1(2^{-k}\sqrt{\sigma^2+|\tau|^2})
\vs_0(2^{[k/3]-k-1}\sigma).
\end{align*}
so that 
$$\zeta_k  ( (\sigma^2+\tau^2)^{1/2})
=\widetilde \zeta_{k}+\sum_{0\le l<k/3}\zeta_{k,l}
.$$

Let $\mu_{k,l}$ be defined as in \eqref{mukla} but with $\zeta_k  ( (\sigma^2+\tau^2)^{1/2})$
replaced by $\zeta_{k,l}$ when $l<k/3$ and by $\widetilde \zeta_k$ when $l=[k/3]$.
We shall prove the following refined version of Proposition \ref{dyadicq-Sob}.
% Then we shall prove
\begin{proposition} \label{muklestprop}
Let $\eps>0$ and $2\le q< \frac{4n+2}{2n-1}$. Then there is $C_\eps>0$ such that for $0\le l\le [k/3]$,
$$\|f*\mu_{k,l}\|_q\le C_\eps 2^{-k \frac{2n-1}q} 2^{l (\eps+\frac {2n}{q}-\frac {2n-1}2)} \|f\|_q.$$
\end{proposition}
If $\frac{4n}{2n-1}<q<\frac{4n+2}{2n-1}$ and $\eps>0$ is sufficiently small then we can sum in $l$ and obtain part (i) of Proposition \ref{dyadicq-Sob}.
The $L^p$ inequality for $p<\frac{4n}{2n+1}$ follows by duality and the estimate for 
$\frac{d}{ds} f*\dil_s\mu_k$ is  proved similarly.

\medskip

{\it Remark.} We also have, by interpolation with an easy $L^\infty$ estimate,
\begin{align}
\|f*\mu_{k,l}\|_q\le C_\eps 2^{-k \frac{2n-1}q} 2^{-l \frac{1-\eps}q} \|f\|_q, \quad \frac {4n+2}{2n-1}\le q\le \infty.
\end{align}

 \section{\it Background and idea of the proof}  The idea in the proof of Proposition 
 \ref{dyadicq-Sob} 
 is to consider the fibers of the fold surface which curved  varying cones; this goes back to the paper \cite{GS} by Greenleaf and one of the authors which dealt with $L^2\to L^p$ inequalities for classes of generalized Radon transforms. One  then would like to apply decoupling for localizations to plates adapted to neighborhoods of these cones. The  cones vary with the base points and some approximation and preparations via changes of variables have  to be used, \cf.   \cite{pr-se-var}.

Concretely if  $\chi_1(x,u)$ and $\chi_2(y,v)$ are compactly supported $C^\infty_c$ functions 
we want to examine the functions  $ ( f\chi_2)*\mu_{k,l} (x,u) \chi_1(x,u)$ 
which are written   in the  form
$$\int \cK_{k,l}(x,u,y,v) f(y,v) dy\, dv$$ where the Schwartz kernel $\cK_{k,l}$ is given by 
\[\chi_1(x,u)\chi_2(y,v) 
 %\chi(u-u^\circ)
 %\chi_0(2^l(x,u-u^\circ)) 
 \iint \zeta(2^{l-k}\sigma)\chi_1(2^{-k} \sqrt{\sigma^2+\tau^2})
 e^{i  \vphi(\sigma,\tau,x,u,y,v)} d\sigma d\tau
  \]
and the phase function is defined by 
 \Be\label{vphidef}
 \vphi(\sigma,\tau,x,u,y,v)=
 \sigma(x_{2n}-y_{2n}-g(\ux-\uy))+\tau(u-v+ x^\intercal \!Jy)
 \Ee
where $J$ is a skew symmetric nondegenerate $2n\times 2n$ matrix 
(for example a 
%$\begin{pmatrix} 0&-I\\I&0\end{pmatrix}$ or a 
skew symmetric perturbation of the standard symplectic matrix). Note we do not assume that $J$ is orthogonal.

With $\vphi(\sigma,\tau, x,u,y,v)$ as in \eqref{vphidef} the cones in question are given, 
for each $(x,u) $, by
\begin{align}
\notag
%\varSigma_{(x,u)}=
&\{(\vphi_x, \vphi_u):\,\, \sigma=0, \, x_{2n}-y_{2n}-g(\ux-\uy)=0,\, 
v=u+x^\intercal Jy \}
\\
\notag
\label{cone-family}
=&\{ (\tau Jy, \tau): \,\,  y_{2n}=x_{2n}-g(\ux-\uy) \} 
%:= \varSigma_x
\end{align}
which is actually independent of $u$.  Denote this conic surface by $\Sigma_x$ and let 
%We shall also set 
\Be \label{Gammaxdefinition}
\Gamma^x(\uy)= (\uy, x_{2n} -g(\ux-\uy)).
\Ee
Then
\Be \label{varSigmacone}
\varSigma_x =\{(\tau J\Gamma^x(\uy),\tau)\}.
\Ee
We wish to use the decoupling inequalities in \cite{bourgain-demeter} (or  the previous paper \cite{LaWo}  if $n$ is sufficiently large)
for thin neighborhoods of the cones  $\Sigma_x$, for suitable frozen $x$. Note that by our assumptions on $g$ the cones are maximally curved (i.e. $d-2=2n-1$  principal curvatures are nonzero). The basic decoupling step will be described in the next section.

% \subsection {Decoupling results for conic manifolds}\label{cc-decoupling} \color{red} Cite the precise quanitative version of the Bourgain-Demeter result.\color{black}
 %This is supposed to contain the result, with proof, of the precise decoupling result  that is needed in the next section.  And we need it in a form that is citable below. Ideally just derive it from the Bourgain-Demeter result for the paraboloid in $\bbR^{d}$,   following the rescaling and induction on scales arguments in \cite{prs} and \cite{bourgain-demeter}.  We need the result for the cones \eqref{varSigmacone} with uniformity in $g$,  $J$ as both are ranging over suitable  compact sets of $C^4$ (or $C^5$, or $C^6$?) functions and antisymmetric nondegenerate matrices, respectively. This does not necessarily have to be included here, if it can be quoted from another paper. \color{black}
 
%\subsection{\it Integration by parts}

% \section{$L^q$-Sobolev estimates and restricted maximal operators}  

 \section{The decoupling step} \label{decouplingstepsect} 
 Let $\delta_0>2^{-l}$ and let $\delta_1<\delta_0$ be such that
%\Be
%\label{delta0-1} \begin{aligned} &\delta_1\ge (2^{-l}\delta_0)^{1/2}, &&\text{ if }\delta_0\le 2^{-l/3},\\&\delta_1\ge \delta_0^2, &&\text{ if } 2^{-l/3}<\delta_0<1.\end{aligned}\Ee

\Be \label{delta0-1} \delta_1\ge (2^{-l}\delta_0)^{1/2} \Ee

Fix $a\in \bbR^{2n}$, $\ub\in \bbR^{2n-1}$.
%, and let $a=(\ua, g(a))$
Suppose  we are given  a family of disjoint cubes 
 $\{Q_\nu\}$ in $\bbR^{2n-1}$ of side length $\delta_1$ contained in the  reference cube 
 \Be\label{refcube}Q:=\{\uy\in \bbR^{2n-1}: |y_i-b_i|\le \delta_0, i=1,\dots 2n-1\}.\Ee
 Suppose in what follows that  for each $\nu$ the function $(y,v)\mapsto f_\nu(y,v) $ is supported in $Q_\nu\times \bbR\times \bbR$.
%We consider the operator $T$ given by 

We fix $\eps>0$  and let 
%\begin{subequations}
%\begin{align}
%\label{l-k}&l\le k/3,\\
\Be \label{delta1-ell}
\delta_1\ge 2^{-l(1-\eps)}.
\Ee
%\end{align}\end{subequations}
 Let $\vphi$ be as in \eqref{vphidef}

% where $\nabla_\ux g(0)=0$.
Let $\chi_{l,a, u^\circ} $ be a smooth function supported in a ball of sidelength $2^{-l}$ centered at $(a, u^\circ)$, satisfying $|\partial^\alpha \chi_{l,a,u^\circ} |\le C_\alpha 2^{l |\alpha|}$ for all multiindices $\alpha$. Let $\zeta$ be a smooth function supported in $(-2,2)$. Let 
$K=K_{k,l,a,u^\circ}$ be given by 
\begin{subequations}
\begin{multline}  
 %\begin{align}\notag
 K(x,u,y,v) \\
  =\chi_{l,a, u^\circ} (x,u)
 %\chi(u-u^\circ)
 %\chi_0(2^l(x,u-u^\circ)) 
 \iint \zeta(2^{l-k}\sigma)\chi_1(2^{-k} \sqrt{\sigma^2+\tau^2})
 e^{i  \vphi(\sigma,\tau,x,u,y,v)} d\sigma d\tau
  \label{Kkldef}
  \end{multline}
  which after a change of variable (replacing $2^{-k}(\sigma,\tau)$ by $(\sigma,\tau)$) we can write
  \Be
  K(x,u,y,v) = %\chi_{l,a, u^\circ} (x,u)\,
 2^{2k}
 \iint 
 %\zeta(2^{l}\sigma)\chi_1(\sqrt{\sigma^2+\tau^2})
 \gamma(\sigma,\tau, x,u,y,v)
 e^{i  2^k\vphi(\sigma,\tau,x,u,y,v)} 
 d\sigma d\tau\,
 \label{Kdefin}
\Ee
\end{subequations} 
with  \begin{subequations}
\begin{align} &\ga(\sigma,\tau, x,u,y,v)=0 \text{ if $|x-a|+|u-u^\circ|+|\sigma|\gc 2^{-l}$}
 \\ 
&|\partial^{M}_{\sigma,\tau, x,u,y,v}\gamma| \le C_M 2^{lM}.
\end{align}
\end{subequations}
Here  $\partial ^M_{\dots}$ stands for any differentiation of order $M$ in the variables indicated.

%\footnote{\color{red} It is better for the induction to assume $|\partial^{|\alpha|}_{\sigma,\tau, x,u,y,v}\gamma| \le  2^{l|\alpha|}$ where $|\alpha|\le M(\eps)$ and $M(\eps)$ is like$(10\eps^{-1})^{100n}$. Differentiation properties with respect to the variables $(y,v)$ are actually better.}

We let $T$ denote any such operator with kernel $K$ and  $\gamma$  as above.
%\color{red} {\it The following proposition has to be quantified  carefully.}\color{black}

\begin{proposition} 
\label{lpdecouplingprop}
Let $2\le q\le \frac{4n+2}{2n-1}$. Let $0<\eps\le 1$, $k\gg 1$, $l\le k/3$, $\delta_1\ge 2^{-l(1-\eps)}$. With the above specifications on $Q$ and $\{Q_\nu\}$ we have, for any $\eps_1\in (0,\eps)$, and $N\in \bbN$, 
\begin{multline}\label{lpdecoupling}
 \Big\| T[\sum_\nu f_\nu]\Big\|_q \le C( \eps_1)  (\delta_0/\delta_1)^{(2n-1)(\frac 12-\frac 1q)+\eps_1}\Big(\sum_\nu\|T  f_\nu\|_q^q\Big)^{1/q}\\ +
 \tilde C(\eps,N) 2^{-k N} \sup_\nu \|f_\nu\|_q.
 \end{multline}
\end{proposition}

 \subsection{\it A model case}\label{model}
 We first consider the model situation 
 $\delta_0$, $\delta_1$ as in \eqref{delta0-1},
$Q$ as in \eqref{refcube},  $\vphi$ as \eqref{vphidef},  
  such that
 \Be\label{nablagzero}a=0, \,\, \ub=0,\,\, \nabla g(0)=0.\Ee
 
% It will be crucial in our proof to  make a quadratic change of variable which modifies $\Phi$  to
 %\Be\label{Psidef}\Psi(\sigma,\tau,x,u,y,v)= \Phi(\sigma,\tau,x,u,y,v)-\tau x_{2n} x^\intercal \!J e_{2n}\,.\Ee
%Let  $\cK_{k,l}$ be as in \eqref{Kkldef} but with $\Phi$ replaced by $\Psi$.
 %Let $\cT=\cT_{k,l}$ be the operator with Schwartz kernel $\cK_{k,l}$. 
% We now have the $\delta_1$-cubes $Q_\nu$ contained in a $\delta_0$-cube centered at the origin in $\bbR^{2n-1}$. We need to prove
  %for $2\le q\le \frac{4n+2}{2n-1}$, and $\eps>0$, $N>0$ 
 %\begin{multline}\label{claim}
% \big\| \cT_{k,l}[\sum_\nu f_\nu]\big\|_q \\
 %\le C(\eps, \eps_1) (\delta_1/\delta_0)^{-\eps} \Big(\sum_\nu\|\cT_{k,l} f_\nu\|_q^2\Big)^{1/2} +
 %\end{multline}
 As pointed out above the   crucial tool is the decoupling estimate from \cite{bourgain-demeter}.
  % see \S\ref{cc-decoupling}.
 The relevant cones $\varSigma_0$ are given by
  %associated with the cone
 \Be\label{Gacone}(\uy,\tau) \mapsto \tau (J\Gamma(\uy),1), \Ee 
 with $ \Gamma\equiv \Gamma^0$ as in \eqref{Gammaxdefinition}, i.e. 
 %\Be \label{coneatorigin} \varSigma_0=  \{(\tau J\Gamma(\uy),\tau)\}, \text{ with } 
\Be \Gamma(\uy)= (\uy, -g(-\uy)).
\Ee
% $\Gamma_i(\uy)=y_i$, $i=1,\dots, 2n-1$ and $\Gamma_{2n} (\uy)= -g(-\uy)$. 
Let
 \Be \label{Nydef}
 %N(\uy)=\frac{e_{2n}-\sum_{i=1}^{2n-1} \frac{\partial g}{\partial x_i}  (-\uy)}{(1+|\nabla g(\uy)|^2)^{1/2}},$$ 
 N(\uy)=e_{2n}-\sum_{i=1}^{2n-1} {\partial_i g} (-\uy)e_i
 \Ee which is  normal  to  
 $\uy\mapsto \Gamma(\uy)$.
 %= (\uy, -g(-\uy))$.
 Let $\uy^\nu\in Q_\nu$ and let  $N_\nu=N(\uy^\nu)/|N(\uy^\nu)|$.
 %$$N_\nud=e_{2n}-\sum_{i=1}^{2n-1} \frac{\partial g}{\partial x_i}  (-\uy^\nu),$$ the normal vector to 
 %$\uy\mapsto \Gamma(\uy)= (\uy, -g(-\uy))$ at $\uy^\nu$.
 % Let $N_{ J}(\uy) = \frac{J^\#N(\uy)}{|J^\# N(\uy)|}$ where 
  Let  $J^\#$ be the contragredient matrix, i.e.  $J^\#=(J^{-1})^\intercal$.  Since $J$ is skew symmetric we have $J^\#= - J^{-1}$.  $J^\# N(\uy)$ is normal to
  %  Note that $N_J (\uy)$ is normal to 
  $J\Gamma$ and 
  $$\inn {J\frac{\partial^2 \Ga(y)}{\partial y_j \partial y_k} }{\frac{J^\#N(\uy)}{|J^\# N(\uy)|} }= \frac {|N(\uy)|}{|J^\# N(\uy)|}
  \inn {\frac{\partial^2 \Ga(y)}{\partial y_j \partial y_k} }{\frac{ N(\uy)}{|N(\uy)|}}.
  $$
  This relates the curvature  form for $\Gamma$ to the curvature form for $J\Gamma$, and  the decoupling estimates from \cite{bourgain-demeter}, 
   in the version for general curved cones (\cite{pr-se}), 
  are applicable.
 %Define $\Pi_\nu(\delta_1)$ as the plate of all $(\oxi,\xi_{2n+1})$ such that
 %\begin{align*}&  \big|\inn{\oxi}{\Gamma(\uy)}+\xi_{2n+1}\big|\le C 2^k \\
% &\Big|\xi_i+\xi_{2n} \frac{\partial g}{\partial x_i} (-\uy^\nu)\Big|\le C2^k\delta_1,\quad i=1,\dots, 2n-1.
% \\ &\Big|\inn{\oxi}{N_\nu}- \xi_{2n+1} \inn{\Gamma(\uy^\nu)}{N_\nu}\Big| \le C 2^k \delta_1^2 \end{align*}
 %We consider the image of the cone under $\begin{pmatrix} J&0\\0&1\end{pmatrix}$.
 
 It turns out that, in order to perform the decoupling step via the Bourgain-Demeter inequality we will have to make a  change of variable in the $ (x,u)$ variables, using  a quadratic shear transformation. Thus we  consider instead the operator
 $\cT$ defined by
 $$\cT f(x,u)= Tf (x, u+ \tfrac 12 \inn{Sx}{x})$$ 
 where $S$ is a suitable symmetric linear transformation. Obviously, by changing variables,  
 \eqref{lpdecoupling} holds with $T$ if and only if it holds with $\cT$.
 
 It will be important in the proof  to choose $S$ such that  the following crucial assumption
  %\begin{assumption} 
  \Be \label{S-assumpt} S J^\# e_{2n} = - e_{2n}
  \Ee is satisfied. 
  To see
  % \footnote{This paragraph may be superfluous and may eventually be deleted}
  that $S$ can be chosen in smooth dependence on $J$ we notice that
$  u_1:=e_{2n}$ and $u_2= J^\#e_{2n}/|J^\#e_{2n}|$ form an orthonormal basis on 
 $\bbV=\text{span}(J^\# e_{2n}, e_{2n})$ which can be extended to an orthonormal basis  
 $\{u_1,\dots, u_{2n}\}$ of $\bbR^{2n}$. Let $c= |J^\# e_{2n}|$ 
 and we  let
  $Su_2=-c^{-1} u_1$, $Su_1= -c^{-1} u_2$ and $Su_i=u_i$ for $i=3,\dots, 2n$. Then $S$ is symmetric, invertible and $$\min\{1,  |J^\# e_{2n}|^{-1}\}\le \|S\|\le \max\{1,  |J^\# e_{2n}|^{-1}\}$$
 where $\|S\|$ denotes the $\ell^2\to \ell^2$ operator norm on $2n\times  2n$ matrices.  
 %$S$ as the identity on the orthogonal complement
% Note  that it is easy to construct such a symmetric matrix. As $J^# e_{2n}$ and $e_{2n}$ are perpendicular we define $S J^# e_{2n} = - e_{2n}$ and $Se_{2n}
 
 The Schwartz kernel of $\cT$ is given by 
 \Be
   2^{2k}
 \iint 
 %\zeta(2^{l}\sigma)\chi_1(\sqrt{\sigma^2+\tau^2})
 \gamma_1(\sigma,\tau, x,u,y,v)
 e^{i  2^k\Phi(\sigma,\tau,x,u,y,v)} 
 d\sigma d\tau\,
 \label{Kmoddefin}
\Ee
with  
\begin{align*} & \ga_1(\sigma,\tau, x,u,y,v)
=\gamma(\sigma,\tau, x,u+ \tfrac 12\inn{Sx}{x} ,y,v)
 \\ 
&\Phi(\sigma,\tau,x,u,y,v) =  
\vphi(\sigma,\tau,x,u,y,v) + \frac{\tau}{2}\inn{Sx}{x}
\end{align*}

 We now define nonisotropic cylinders (or  ``plates") associated to the cone \eqref{Gacone}.
 We use the notation 
 $$\xi=(\overline\xi,\xi_{2n+1})=(\underline\xi, \xi_{2n}, \xi_{2n+1}).$$
  The tangent space to the cone at $\tau(J\Gamma(\uy^\nu),1)$ is spanned
 by $$\{ J\Ga(\uy^\nu),1\}\cup\{(J \partial_i \Gamma (\uy^\nu),0): i=1,\dots, 2n-1\},$$
 and a normal vector is given by
 \Be \label{Jnormal} \begin{pmatrix}
 J^\#\!N_\nu  \\ - \inn {J\Gamma(\uy^\nu)}{J^\# N_\nu}
 \end{pmatrix}
 %e_{2n+1}  
 =\begin{pmatrix}  J^\#N_\nu\\- \inn{\Ga(\uy^\nu)}{N_\nu}\end{pmatrix}.
 %e_{2n+1}.
 \Ee
  The relevant plates  are $2^k \Pi_{\nu}(\delta_1)$
  where  the normalized plates  $ \Pi_{\nu}(\delta_1)$ are  defined by the inequalities
  \begin{subequations} \label{platedefinition}
  \begin{align}
%&  \big|\inn{\oxi}{J\Gamma(\uy)}+\xi_{2n+1}\big|\le C 2^k
&C^{-1} \le \sqrt{|\oxi|^2+\xi_{2n+1}^2}\le C  
 \\
 %&\Big|\inn{\oxi- \xi_{2n+1}J\Ga(\uy^\nu)}{\frac{\partial J\Gamma}{\partial y_i} (\uy^\nu)}\Big|\le C2^k\delta_1,\quad i=1,\dots, 2n-1.
 &|\oxi-\xi_{2n+1} J\Ga(\uy^\nu)| \le C \delta_1\label{plateb}
 \\
 &\big|\inn{\oxi -\xi_{2n+1} J\Ga(\uy^\nu)}{J^\#\!N_\nu}\big| \le C \delta_1^2.\label{platec}
 \end{align}
 \end{subequations}
 
The Bourgain-Demeter decoupling theorem 
%(with an  extension to general cones with curvature as in \cite{pr-se}) 
%\footnote{\color{red} Expand?\color{black}}
gives that 
%\footnote{This should be properly formulated in \S\ref{cc-decoupling}}
$$\Big\|\sum_\nu  F_\nu\Big\|_q \le C(\eps)  (\delta_1/\delta_0)^{-\eps}  \Big(\sum_\nu \|F_\nu\|_q^2\Big)^{1/2}, 
\quad 2\le q\le \frac{4n+2}{2n-1}\,,
$$
provided that the Fourier transforms $\widehat F_\nu$ are supported in $2^k\Pi_{\nu}(\delta_1)$. 
%Moreover the constant $\delta_1^{-\eps}$ can be replaced with $(\delta_1/\delta_0)^{-\eps}$ if $\delta_1<\delta_0<1$ and the $\uy^\nu$ are taken from a ball of diameter $\delta_0$. 
We have some freedom in the choice of the constant $C$ ranging over a compact subset of $(0,\infty)$. 
%\footnote{\color{red} Be more specific?\color{black}}
%If we change the constant $C$ the bounds depend in a continuous way on $C$
%(or a fixed multiple of these plates).
Let $\eta_\nu$ be a bump function which is equal to $1$ on $\Pi_{\nu}(\delta_1)$ and is supported 
on its double, and $\eta_\nu$ satisfies the natural differential inequalities. Specifically consider the radial tangential, nonradial tangential, and normal differentiation operators:
%if we define the first order differential operators
\begin{align*}
\cV_{\nu,0}&=\inn{J\Gamma(\overline y^\nu)}{\nabla_{\overline \xi}}+ \frac{\partial}{\partial\xi_{2n+1}},
\\
\cV_{\nu,i} &= \frac{\partial}{\partial \xi_i} - \inn{ J\Gamma(\overline y^\nu)}{e_i} \frac{\partial}{\partial \xi_{2n+1}},\quad i=1,\dots, 2n-1,
\\
\cV_{\nu}&= \inn{J^\# N_\nu}{\nabla_{\overline \xi} }- 
\inn{\Gamma(\overline y^\nu)}{N_\nu}\frac{\partial }{\partial\xi_{2n+1}}.
\end{align*}
Then
\[
\big|
\cV_{\nu,0}^{\alpha_0} \cV_{\nu,1}^{\alpha_1}\dots \cV_{\nu,2n-1}^{\alpha_{2n-1}} \cV_{\nu}^{\alpha_{2n}}\eta_\nu(\xi)\big|\,\le\,
C_{\alpha_0,\alpha_1,\dots, \alpha_{2n}} (1/\delta_1)^{2\alpha_{2n}+\sum_{i=1}^{2n-1}\alpha_i}.
%\delta_1^{-\alpha_1-\dots -\alpha_{2n-1}-2\alpha_{2n}}\,.
\]
%\footnote{\color{red} the usual, but specify \color{black} }
Define the Euclidean convolution operator $P_{k,\nu}$ in the multiplier formulation by  $$\widehat {P_{k,\nu} f}(\xi)=\eta_\nu(2^{-k}\xi)\widehat f(\xi).$$
Then by the decoupling inequality
\begin{multline}\label{decoupling} \Big\|\sum_\nu  P_{k,\nu} \cT f_\nu\Big\|_q \le C(\eps)  (\delta_1/\delta_0)^{-\eps}  \Big(\sum_\nu \|\cT f_\nu\|_q^2\Big)^{1/2}, \\
\quad 2\le q\le \frac{4n+2}{2n-1}\,.
\end{multline}
We need to  analyze the Schwartz kernel of $ f\mapsto (I-P_{k,\nu})\cT$ when acting on $f_\nu$.
%\begin{multline*} (I-P_\nu)T_\nu f_\nu(x)= \int \iint e^{i\inn{x-\tx}{\oxi} - i(u-\tu)\xi_{2n+1}} (1-\eta_\nu(2^{-k}\xi)) \chi_0(2^l(\tx, \tu-u^\circ))\\ \times  \iint e^{i\Psi(\sigma, \tau, \tx, \tu,y,v)}  \chi_0(2^{l-k}\sigma) \chi_1(2^{-k}\sqrt{\sigma^2+\tau^2}) d\sigma d\tau\, d\tx d\tu\, d\xi\end{multline*}
Thus we consider for $y\in Q_\nu$, 
\begin{multline*} 2^{(2n+3)k}
 \int \iint e^{i2^k\inn{x-\tx}{\oxi} + i2^k(u-\tu)\xi_{2n+1}} (1-\eta_\nu(2^{-k}\xi))
% \chi_{l,0} (\tx,\tu)
\\ \times  \iint e^{i2^k\Phi(\sigma, \tau, \tx, \tu,y,v)} 
% \chi_0(2^{l}\sigma) \chi_1(\sqrt{\sigma^2+\tau^2}) 
\ga_1(\sigma,\tau, \tx, \tu, y,v) d\sigma d\tau\, d\tx d\tu\, d\xi\,.
\end{multline*}

We can replace $(I-P_{k,\nu})\cT  f_\nu$  with  
$(I-P_{k,\nu} ) L_k \cT  f_\nu$  with  $L_k$ a Littlewood-Paley cutoff operator localizing to frequencies 
$C_1^{-1}  2^k\le |\xi| \le C_1 2^k$ (as the remaining error terms are handled by standard integration by part arguments, see e.g.\cite[\S4.2]{hoer}).
%\footnote{\color{red} say more? \color{black}}.
Also if $|x|+ |u| >C$ we see that by an additional integration by parts arguments in the 
$(\xi, \sigma,\tau)$ variables we get a bound 
$\lc C 2^{-(k-l)N}  (|x|+|u|)^{-N}$.
Hence we need to show that the $L^\infty$ norm of the integral is $O(2^{-kN})$. Since the $(\tx, \tu)$ integral is over a compact set it suffices to show that the Fourier transform of $\cT f_\nu$ in the complement of the double plate has norm $O(2^{-kN})$. Moreover, if
$|\tau-\xi_{2n+1}|\ge 2^{l-k}$ then we can integrate by parts with respect to $\tu$ and show that the resulting integral is $O(2^{-kN})$.
%\footnote{\color{red} All of this is standard, but  we need to carry out some details,to keep track on the dependence on $g$ and the symbols etc.\color{black}}

It remains to analyze, for $|\tau|\approx |\xi_{2n+1}|$  and $y\in Q_\nu$, the behavior of 
\Be \label{oscint}
 \iiint
   e^{i2^k(\Phi(\sigma, \tau, x, u,y,v) -\inn{x}{\oxi}-u\xi_{2n+1}) }
   %\chi_{l,0}(x,u)\chi_0(2^{l}\sigma) \chi_1(\sqrt{\sigma^2+\tau^2}) 
 \gamma_1(\sigma,\tau,x,u,y,v)   d\sigma  \,dx \, du \Ee
assuming that $\xi\notin \Pi_\nu(\delta)$ with $C$  in \eqref{platedefinition}  large.
% in the definiition of $\Pi_\nu(\delta)$
For   better readability we have changed  the notation from 
$(\tx,\tu)$ to $(x,u)$ in \eqref{oscint}.
In what follows we will set 
\Be\label{Psidef} \Psi(\sigma, \tau, x, u,y,v,\xi)= \Phi(\sigma, \tau, x, u,y,v) 
-\inn{x}{\oxi}-u\xi_{2n+1}\,.
\Ee

In order to estimate Fourier transforms in the complement of plates we need bounds for certain directional derivatives of the phase function $\Phi$ (which will then turned into lower bounds for the directional derivatives of $\Psi$ when $\xi$ is away from the plate).

\begin{lemma} \label{derivPhi} 
There is a constant $A\ge 1$ so that the following statements hold for all $y\in Q_\nu$, 
$|\sigma|\approx 2^{-l}$, $|\tau|\approx 1$.

(i) Let 
$\vec V_{\nu,i}= e_i-\inn{J\Gamma(\uy^\nu)}{e_i} e_{2n+1}.$ 
Then
\Be \label{upperbd-dirderiv1}
\big|\inn {\vec V_{\nu,i}}{\nabla_{x,u} \Phi} 
\big|
\le A
\big( \delta_1+ |x_{2n}-y_{2n}-g(\ux-\uy)|\big). \, \, 
\Ee

(ii) Let $\vec V_{\nu}$ be the normal vector in \eqref{Jnormal}. Then
\Be \label{upperbd-dirderiv2}
\big|\inn {\vec V_{\nu}}{\nabla_{x,u} \Phi} 
\big|
\le A \big(\delta_1^2+ |x_{2n}-y_{2n}-g(\ux-\uy)|\big). \, \, 
\Ee
\end{lemma}

\subsubsection{Proof of Lemma \ref{derivPhi}}
To see (i) we have for $i=1,\dots, 2n$,
\begin{align*}
\inn {\vec V_{\nu,i}}{\nabla_{x,u} \Phi} 
&=
\frac{\partial \Phi}{\partial x_i} - \inn{J\Gamma(\uy^\nu)}{e_i} 
\frac{\partial \Phi}{\partial u}
\\
&=\sigma\inn{N(\uy)}{e_i}+\tau (\inn{Jy}{e_i}- \inn{J\Gamma(\uy^\nu)}{e_i} + \inn{Sx}{e_i}).
%\le C_2\big(\delta_1+ |x_{2n}-y_{2n}-g(\ux-\uy)|\big). \, \, 
%\text{ if } y\in Q_\nu \,.
%\text{ and } |y_{2n}-x_{2n}-g(\ux-\uy)|\le 2^{-l}.
\end{align*}

We have $\sigma=O(2^{-l})$, $Sx=O(2^{-l})$ and 
\begin{align*}Jy-J\Gamma(\uy^\nu)&= 
%(y_{2n}-g(-\uy^\nu))\inn{e_i}{e_{2n}}+ O(\delta_1)\,.
(y_{2n}+g(-\uy)) Je_{2n} + (g(-\uy^\nu)-g(-\uy)) Je_{2n}
\\
&=(y_{2n}+g(-\uy)) Je_{2n} + O(\delta_1)
\end{align*}
Split
\begin{align*} 
|y_{2n}+g(-\uy^\nu)|&\le |y_{2n}-x_{2n}+g(\ux-\uy)| + |x_{2n}|
\\
&+ |g(-\uy)-g(\ux-\uy)|+|g(-\uy^\nu)-g(-\uy)|.
\end{align*}
%The first term on the right hand side is $O(2^{-l})$ by assumption 
%\eqref{upperbd-dirderiv1}. 
The second and the third terms are $O(2^{-l})$, by our localization in $x$. The fourth term is $O(\delta_0\delta_1)$ since $\uy\in Q_\nu$ and $\nabla g =O(\delta_0)$  in $Q_\nu$. Consequently
\Be\label{auxest}
|y_{2n}+g(-\uy^\nu)|
\lc (2^{-l} +\delta_0\delta_1)+ |x_{2n}-y_{2n} -g(\ux-\uy)|  
\Ee
and thus \eqref{upperbd-dirderiv1}  follows easily.

Next we verify  (ii).
We have
\begin{align}\label{JNder}
\inn {\vec V_{\nu}}{\nabla_{x,u} \Phi}& =
\inn{J^\# N_\nu}{\nabla_x\Phi} 
- \inn{J^\#N_\nu}{J\Gamma(\uy^\nu)} 
\frac{\partial\Phi}{\partial u}
\\\notag &=I (x,y,\sigma)+II(y,\sigma)+III(x,y,\tau)+IV(x,\tau)
\end{align}
where
\begin{align*}
I(x,y,\sigma)&= - \sigma \sum_{i=1}^{2n-1} \inn {J^\# N_\nu}{\partial_i g(\ux-\uy)-\partial_i g(-\uy)},
\\
II(y, \sigma) &= - \sigma \sum_{i=1}^{2n-1} \inn {J^\# N_\nu}{\partial_i g(-\uy)} + \sigma\inn {J^\# N_\nu}{e_{2n}},
\\
III(x, y,\tau)&= \tau\inn{J^\# N_\nu}{Jy} - \tau \inn{N_\nu}{\Gamma(\uy^\nu)},
\\IV(x,\tau)&=  \tau \inn {J^\# N_\nu}{Sx}.
\end{align*}
Since $\sigma=O(2^{-l})$  and $|x|=O(2^{-l})$ we have 
$| I(x,y,\sigma)| \lc 2^{-2l}$ which is of course $O(\delta_1^2)$.
By \eqref{Nydef} we have 
\begin{align*} II(y,\sigma)&= \sigma \inn{ J^\# N_\nu}{N(\uy)}
\\ 
 &=\sigma \inn{ J^\# N_\nu}{N(\uy)- N(\uy^\nu)} + \sigma |N(\uy^\nu)| \inn{J^\#N_\nu}{N_\nu}
 \end{align*} and by the skew symmetry of $J^\#$ the last summand drops out and we get 
%\label{JNder} &=I+II+II\sigma \inn{J^\# N_\nu}{N(\uy)} +\tau(\inn {J^\#N_\nu}{\nabla_x\Phi} -  \inn{J^\#N_\nu}{J\Gamma(\uy^\nu)} )
%%%%%%\\\lc C_1\delta_1^2 +C\delta_0|x_{2n}-y_{2n}-g(\ux-\uy)|\quad\text{ if } y\in Q_\nu .\end{align}
%Expand $N(\uy)$ about $\uy^\nu$ and use the skew symmetry of $J^\#$ to estimate
$\sigma \inn{J^\# N_\nu}{N(y)} =O(2^{-l}\delta_1)$ which is   $O(\delta_1^2)$.
For the third term, we write
\begin{align}
\notag
\tau^{-1} III(x,y,\tau)&= \inn {N_\nu}{y- \Gamma(\uy^\nu)} 
%+ \inn{J^\# N_\nu}{Sx}\big)
\\
%\label{threeterms}
\notag
&= \inn {N_\nu}{\Gamma(\uy)- \Gamma(\uy^\nu)} 
+ \inn{N_\nu}{y-\Ga(\uy)}.
%%(y_{2n}+g(-\uy)) \inn{N_\nu}{e_{2n}}
%+\inn{N_\nu}{J^{-1} Sx}\,.
\end{align}
%\begin{align}\notag&\tau(\inn {J^\#N_\nu}{\nabla_x\Phi} -  \inn{J^\#N_\nu}{J\Gamma(\uy^\nu)} )\\ &=\inn{J^\# N_\nu}{J\Gamma(y) -J\Gamma (y^\nu)} +(y_{2n}-g(-\uy)) \inn{J^\# N_\nu}{e_{2n}} \end{align}
By definition of $N_\nu$ and Taylor expansion,
\[ \inn {N_\nu}{\Gamma(\uy)- \Gamma(\uy^\nu)} = O(|\uy-\uy^\nu|^2)= O(\delta_1^2).\]
Next observe $N_\nu=e_{2n}+ O(\delta_0)$ and 
\begin{align*}x_{2n}-  g(x-\uy) +g(-\uy) &= x_{2n}- \inn{\ux}{\nabla g(-\uy)}+ O(2^{-2l})
\\&= \inn {x}{N(y)} \,+ O(2^{-2l})
\end{align*} and thus 
\begin{align*} 
&\inn{N_\nu}{y-\Ga(\uy)}=  \inn{N_\nu}{e_{2n}} (y_{2n}+ g(-\uy))
\\
=&- \inn{N_\nu}{e_{2n}} (x_{2n}- y_{2n}- g(\ux-\uy))  
+\inn{N_\nu}{e_{2n}} (x_{2n}-  g(x-\uy) +g(-\uy))
\\
=&
\inn{N_\nu}{e_{2n}} \inn{x}{N(y)} 
+ O\big (2^{-2l}+|x_{2n}- y_{2n}- g(\ux-\uy)|\big),  
\end{align*}
and furthermore 
\begin{align*}
\inn{N_\nu}{e_{2n}} \inn{x}{N(y)} &= \inn {N(y)} {x}(1+O(\delta_0)) \\&= 
\inn {N_\nu}{x} 
+  O( 2^{-l} \delta_0)).
\end{align*}
Adding $\tau^{-1} IV(x,\tau)$ we get
%We incorporate the third term in \eqref{threeterms} and get
\begin{align*}
&\inn {N_\nu}{x} +\tau^{-1} IV(x,\tau)= 
%\inn {N_\nu}{J^{-1} Sx} = 
\inn {(I+SJ^\#) N_\nu}{x}
\\ 
&=\inn {(I+SJ^\#) e_{2n}} {x} + O(2^{-l}\delta_0) \,=\, O( 2^{-l}\delta_0)
\end{align*}
where in the last equation we have used the crucial assumption  \eqref{S-assumpt} on the choice of $S$, i.e. that $e_{2n}$ is in the nullspace of $I+SJ^\#$.
Collecting these estimates we obtain 
\[|\tau|^{-1} | III(x,y,\tau)+ IV(x,\tau) | \lc  (\delta_1^2+ 2^{-l}\delta_0 + |x_{2n}- y_{2n}- g(\ux-\uy)|).
\]
By our assumption 
 \eqref{delta0-1} 
 we have $2^{-l}\delta_0\lc \delta_1^2$ and 
 the asserted estimate in (ii) follows.
 %. If $\delta_0<2^{-l/3}$ then we  assume $\delta_1\ge 2^{-l/2}\delta_0^{1/2}$  and thus  also $\delta_1\ge \delta_0^{3/2}\delta_0^{1/2} =\delta_0^2$. If $\delta_0\ge 2^{-l/3} $ we assume $\delta_1\ge \delta_0^2$ which also implies $\delta_1\ge 2^{-l/2}\delta_0^{1/2}$ in this range.
% Hence we get  in both cases
%$$\delta_0(\delta_0\delta_1 +2^{-l})\le \delta_1^2.$$
%It is the above estimate  that  necessitated the conditions on $\delta_1$  in \eqref{delta0-1}.
The proof is complete. \qed

\subsubsection{Estimation of Fourier transforms in the complement of the plates}
We apply Lemma \ref{intbyparts} in a two-dimensional setting 
where the $w_1$-derivative will be replaced with the directional derivative for a vector
$\vec V$ in $\bbR^{2m+1}$ and where $w_2=\sigma$.

We first assume that \eqref{plateb} does not hold, i.e. 
\Be\label{platenotb} 
|\xi_i-\inn{J\Gamma(\uy^\nu)}{e_i} \xi_{2n+1}
| \ge C_1\delta_1 
%\text{ for some $i\in \{1,\dots, 2n\}$.}
\Ee
for some $i\in \{1,\dots, 2n\}$ and $C_1>2A$, with $A\ge 1$ as in Lemma \ref{derivPhi}.

%Here $C_1$ will be chosen large, to be specified after  the upcoming discussion.
%We work with the direction $$\vec V_{\nu,i}= e_i-\inn{J\Gamma(\uy^\nu)}{e_i} e_{2n+1}.$$
% for an  index $i$ for which \eqref{platebotb} holds.
%We claim that for all $y\in Q_\nu$  \footnote{\color{red} For expository reasons it will be better to first formulate  \eqref{upperbd-dirderiv1} and \eqref{upperbd-dirderiv2} as a lemma and then choose $C_1\gg C_2$ in \eqref{platenotb}, \eqref{platenotc}.}

Note that $\frac{\partial \Phi}{\partial \sigma}= x_{2n}-y_{2n}- g(\ux-\uy)$. Hence if 
\eqref{platenotb} holds with $C_1 \ge 2A \ge 2$ then, 
%with $\nabla\equiv \nabla_{x,u}$, 
we get from part (i) of Lemma \ref{derivPhi}
\Be \label{lowerbdgradient}
\big | \inn{\vec V_{\nu,i}}{\nabla_{x,u}\Psi}\big|+ \big|\frac{\partial \Psi}{\partial \sigma}\big|\ge \delta_1/2.
\Ee Indeed the left hand side is equal to 
\begin{align} \notag 
&\Big|\frac{\partial \Phi}{\partial x_i} - \inn{J\Gamma(\uy^\nu)}{e_i} 
\frac{\partial \Phi}{\partial u} -\xi_i + 
\inn{J\Gamma(\uy^\nu)}{e_i} \xi_{2n+1}
\Big| + \Big|\frac{\partial \Phi}{\partial \sigma}\Big|
\\  \notag
& \ge \max \{ 0, (C_1-A) \delta_1 - A |x_{2n}-y_{2n} -g(\ux-\uy)|\}  +  |x_{2n}-y_{2n} -g(\ux-\uy)| 
\\ \notag&\ge \frac{ C_1-A}{2A} \delta_1 \ge \frac{\delta_1}{2} .
\end{align}
We now use integration by parts. 
%Let $\Psi$ be as in \eqref{Psidef}  and   $\nabla\equiv \nabla_{(x,u)}$.
%$$\Psi(\sigma, \tau, x,u,y,v, \xi)=\Phi(\sigma,\tau, x,u, y,v)-\inn{x}{\oxi}-u\xi_{2n+1}.$$
We assume \eqref{platenotb} and define
%, with  $\vec V_{\nu,i}= e_i- \inn{J\Gamma(\uy^\nu)}{e_i} e_{2n+1}$,   
 a differential operator $L$ by 
%\begin{align*}
\Be
Lh =\inn {\vec V_{\nu,i}}{\nabla} 
\Big( \frac{ \inn {\vec V_{\nu,i}}{\nabla \Psi} 
 \inn {\vec V_{\nu,i}}{\nabla h} }
{| \inn {\vec V_{\nu,i}}{\nabla \Psi} |^2 + |\tfrac{\partial \Psi}{\partial \sigma}|^2}\Big)
+
\frac{\partial}{\partial \sigma} 
\Big( \frac{ \tfrac{ \partial\Psi}{\partial\sigma} \tfrac{ \partial h}{\partial\sigma} }
{| \inn {\vec V_{\nu,i}}{\nabla \Psi} |^2 + |\tfrac{\partial \Psi}{\partial \sigma}|^2}\Big).
\Ee
%\end{align*}
%If $h(x,u,\sigma,\cdot) = 
%\chi_{l,0}(x,u)\chi_0(2^{l}\sigma) \chi_1(\sqrt{\sigma^2+\tau^2}) $
The integral \eqref{oscint} becomes
\Be \label{oscintnew}
i^N 2^{-kN}  \iiint
   e^{i2^k(\Psi(\sigma, \tau, x, u,y,v, \xi)) } L^N \ga_1(\sigma,\tau, x,u,y,v) d\sigma  dx du \Ee
and $|L^N \ga |\lc_N (2^l\delta_1^{-1})^N$ by a straightforward analysis using Lemma \ref{intbyparts}.  Since 
$2^{k-l}\delta_1 \gc 2^{k/3}$ we gain a factor $2^{-kN_1/3}$ with $N_1$ integrations by parts.

\bigskip

Next consider the  more subtle  case where  \eqref{platec} does not hold, i.e. we have
\Be \label{platenotc}
\big|\inn{\oxi -\xi_{2n+1} J\Ga(\uy^\nu)}{J^\#\!N_\nu}\big| \ge C_1 \delta_1^2 
\Ee provided that  $C_1\ge 2A$.
%Let  $V_\nu$ be  the normal vector defined in \eqref{Jnormal} and$\nabla\equiv\nabla_{(x,u)}$.
We now have by part (ii) of Lemma \ref{derivPhi}
\Be \label{lowerbdgradient2}
\big | \inn{\vec V_{\nu}}{\nabla\Psi}\big|+ \big|\frac{\partial \Psi}{\partial \sigma}\big| \ge \delta_1^2/2.
\Ee
To see this observe that the left hand side is equal to
\begin{align}
 \notag 
&\Big|\inn{J^\#N_\nu}{\nabla_x\Phi} - \inn{\Gamma(\uy^\nu)}{N_\nu} \frac{\partial \Phi}{\partial u}
 -\inn{J^\#N_\nu} {\overline \xi}+ 
\inn{\Gamma(\uy^\nu)}{N_\nu} \xi_{2n+1}
\Big| + \Big|\frac{\partial \Phi}{\partial \sigma}\Big|
\\  \notag
& \ge \max \{ 0, (C_1-A) \delta_1^2 - A |x_{2n}-y_{2n} -g(\ux-\uy)|\}  +  |x_{2n}-y_{2n} -g(\ux-\uy)| 
\\ \notag&\ge \frac{ C_1-A}{2A} \delta_1^2 \ge \frac{\delta_1^2}{2} .
\end{align}

We use for our integration by parts the operator  $\tilde L$  defined by
\Be
\tilde Lh=\inn {\vec V_{\nu}}{\nabla} 
\Big( \frac{ \inn {\vec V_{\nu}}{\nabla\Psi} 
 \inn {\vec V_{\nu}}{\nabla h} }
{| \inn {\vec V_{\nu}}{\nabla \Psi} |^2 + |\tfrac{\partial \Psi}{\partial \sigma}|^2}\Big)
+
\frac{\partial}{\partial \sigma} 
\Big( \frac{ \tfrac{ \partial\Psi}{\partial\sigma} \tfrac{ \partial h}{\partial\sigma} }
{| \inn {\vec V_{\nu}}{\nabla \Psi} |^2 + |\tfrac{\partial \Psi}{\partial \sigma}|^2}\Big).
\Ee

%\end{align*}
Again  we get the formula  \eqref{oscintnew} with $L$ replaced by $\tilde L$ and we need to examine the symbol
$(\tilde L)^N \ga$ using the crucial lower bound \eqref{lowerbdgradient2}.
%By the above estimate
%By the above discussion we can conclude\Be \label{lowerbdgrad}
%$\big(| \inn {\vec V_{\nu}}{\nabla \Psi} |^2 + |\tfrac{\partial \Psi}{\partial \sigma}|^2\big)^{1/2}\gc \delta_1^2$
We use the terminology in the appendix (\cf.  Definition \ref{typeABterms}). 
Analyzing the terms of type $(A,j)$ we thus get a bound $O((2^l\delta_1^{-2})^j)$.
For the terms $(B,1)$  (second derivative of $\Psi$ divided by the square of a gradient) we notice that pure $(\sigma,u)$  derivatives of second order are zero and pure $x$ derivatives of second order carry the factor $\sigma=O(2^{-l})$.
We need to get an upper bound for mixed  derivatives of second order and notice that
\Be \frac{\partial }{\partial\sigma} \inn{\vec V_\nu}{\nabla} \Psi= -\inn{J^\# N_\nu}{N(y)} = O(\delta_1)
\Ee
for $y\in Q_\nu$. This shows that the  type $(B,1)$ terms are $\lc  \delta_1^{-3} +2^{-l} \delta_1^{-4}$ and hence  $\lc 2^l \delta_1^{-2}$. 
The type $(B,j)$ terms for $j\ge 2$ are $O(\delta^{-2(j+1)})$.
Hence if $\beta$ is a product of 
a bounded term, a term of   type $(A, j_A)$,  $M_1$ terms of type $(B,1)$ and 
$M_2$ terms of type 
$(B, \kappa_i) $  with $\kappa_i\ge 2$,  and if $j_A+ M_1+ \sum_{i=1}^{M_2}\kappa_i =N_1$,
we get a bound 
$$|\beta|\lc (2^{l} \delta_1^{-2})^{j_A +M_1}  \prod_{i=1}^{M_2}  \delta_1^{-2(\kappa_i+1)}$$
and hence
\begin{align*}2^{-kN_1} |(\tilde L)^{N_1}  \ga|& \lc_{N_1} 
(2^{k-l}\delta_1^2)^{-j_A-M_1} \prod_{i=1}^{M_2}  \big(2^{-k} \delta_1^{-2\frac{\ka_i+1}{\ka_i}}\big) ^{\kappa_i}
\\
&\lc_{N_1}
 (2^{k-l}\delta_1^2)^{-N_1}\lc
2^{-kN_1} 2^{l(3-2\eps)N_1}\lc 2^{-2k N_1\eps/3}
\end{align*}
 since we are assuming $\delta_1>2^{-l(1-\eps)}$ and $l\le k/3$. We choose 
 $N_1$ large, say $N_1> (2N+10 n)/\eps$,  and 
% and $\delta_1>2^{-l}$.
from \eqref{decoupling} and the above error analysis we get the bound
\begin{multline}\label{l2decoupling} \Big\|\sum_\nu  P_{k,\nu} Tf_\nu\Big\|_q \\ \le C(\eps)  (\delta_1/\delta_0)^{-\eps}  \Big(\sum_\nu \|Tf_\nu\|_q^2\Big)^{1/2}
+C_3(\eps, N)  2^{-kN} \sup_\nu \|f_\nu\|_q.
%\quad 2\le q\le \frac{4n+2}{2n-1}\,.
\end{multline}
Now apply H\"older's inequality in the $\nu$ sum (which has  $\lc (\delta_1/\delta_0)^{-(2n-1)}$ terms) to also get  \eqref{lpdecoupling}. This  finishes the proof of Proposition \ref{lpdecouplingprop}
under the additional assumption \eqref{nablagzero}.

%\subsection {\it Quadratic change of variable} 
%We now let $g$ as in \S \ref{model}, i.e. with $g(0)=0$ and $\Phi$ as in \eqref{Phidef}. Let %$T\equiv T_{k,l}$ be the integral operator with Schwartz kernel 
 %$$K_{k,l} (x,u,y,v)  = \chi(2^lx) \iint e^{i \Phi(\sigma, \tau, x,u, y,v)} \chi_0(2^{l-k}\sigma)  \chi_1(2^{-k} \sqrt{\sigma^2+\tau^2}) d\sigma \, d\tau. $$
% Then  $$T_{k,l} f(x,u)= \cT_{k,l}f (x, u+x_{2n} x^\intercal Je_{2n})$$ and the inequality \eqref{claim} immediately gives the corresponding inequality with $\cT_{k,l}$ replaced by $T_{k,l}$, by applying the change of variable on the left hand side and then the inverse change of variable on the right han side of the inequality.
 
% By H\"older's inequality (and the fact that we have $O((\delta_0/\delta_1)^{2n-1}$ many cubes $Q_\nu$ in the reference cube, we get \begin{multline}\label{claimpp} \big\| T[\sum_\nu f_\nu]\big\|_p \le C(\eps, \eps_1)  (\delta_0/\delta_1)^{(2n-1)(\frac 12-\frac 1p)-\eps}\Big(\sum_\nu\|\cT f_\nu\|_p^p\Big)^{1/p}\\ + \tilde C(\eps_1,N) 2^{-k) N} \sup_\nu \|f_\nu\|_p. \end{multline}

\subsection{\it Changes of variables} \label{changesofvar}
We now complete the proof of Proposition 
\ref{lpdecouplingprop} by reducing the general case to the model case \eqref{nablagzero}.

 Let again $\delta_1, \delta_0$ be as in \eqref{delta0-1}.
 % $\delta_1^2\ge 2^{-l}\delta_0\,.$
We are now given a family of disjoint cubes 
%of sidelength $\delta_1$ contained in a cube of sidelength $\delta_0$, again with 
 $\{Q_\nu\}$ in $\bbR^{2n-1}$ of sidelength $\delta_1$ contained in a reference cube $$Q:=\{\uy\in \bbR^{2n-1}: |y_i-b_i|\le \delta_0, i=1,\dots 2n-1\}$$
 and suppose that the function $f_\nu$ is supported in $Q_\nu\times \bbR\times\bbR$.
We consider the operator $T$ with Schwartz kernel as in \eqref{Kdefin} but do not assume that $\nabla g$ vanishes at the reference point $(\ua,\ub)$. 
Decomposing the cutoff function in $(x,u)$ into a finite number of pieces (with the number depending on upper bounds for $g'$ we may assume that
\Be\supp(\chi_{l,a,u^\circ})\subset \{ |x-a|\le c_02^{-l}, \,|u-u^\circ|\le c_02^{-l}  \}
\label{smallc0}\Ee
for some small $c_0>0$.
We also set
$a=(\ua, a_{2n})$,  and  $b=(\ub,0)$.

%We introduce new variables. 
Define $G\equiv G_{a, \ub}$ by  $$G (\uw)= g(\ua-\ub+\uw)- a_{2n}- \uw^\intercal \nabla_xg(\ua-\ub)$$
so that
$$G(0)= -a_{2n}+g(\ua-\ub), \quad G'(0)=0\,.$$ 
We now  introduce the change of variables $(X,U)=(\uX,X_{2n}, U)$,
$(Y,V)=(\uY,Y_{2n},V)$, defined by 
\begin{align*} \uX&=\ux-\ua, 
\\X_{2n}&= x_{2n}-a_{2n}-(\ux-\ua)^\intercal  \nabla g(\ua-\ub)
\\ U&= u+\la(x)+(x-a)^\intercal Jb 
\end{align*} and
\begin{align*} \uY&=\uy-\ub, 
\\ Y_{2n}&= y_{2n}-(\uy-\ub)^\intercal  \nabla g(\ua-\ub)
\\ V&=v+\la(y)-a^\intercal Jy
\end{align*}
Then  
%setting $G(w)= g(\ua-\ub+w)-a_{2n}-w^\intercal g'(\ua-\ub)$ 
we have 
\begin{subequations}
\Be \label{cofv1} x_{2n}-y_{2n}-g(\ux-\uy)=X_{2n}-Y_{2n}-G(\uX-\uY).\Ee
Setting  $B=\begin{pmatrix} I_{2n-1}&0\\g'(\ua-\ub)&1\end{pmatrix}$ we also have 
\begin{align}\notag u-v+x^\intercal J y+\la(x-y) &= U-V+(x-a)^\intercal J (y-b)
\\
\label{cofv2}&=U-V+X^\intercal B^\intercal JBY.
\end{align}
Here $B^\intercal JB$ belongs to a compact family of invertible skew symmetric $2n\times 2n$ matrices.
\end{subequations}
Now, after a decomposition into a finite number of pieces with $(x,u)$-support of diameter $<2^{-l}$   we can   reduce to the situation  where we can apply the estimates in \S\ref{model}.
%Define $\imath_R$, $\imath_L$ by $\imath^R(X,U)=(x,u)$, $\imath^R(Y,V)=(y,v)$ and let $g_\nu=f_\nu\circ \imath_R$.Then $g_\nu$ is supported in $\tilde Q_\nu\times\bbR\times\bbR$ where $\tilde Q_\nu$ is a $\delta_1$-cube and a translate of $Q_\nu$.Let 

%This completes the proof of Proposition \ref{lpdecouplingprop}.

\section{Proof of Proposition \ref{dyadicq-Sob}}\label{iterationsect} 
 
%\subsection{\it Iteration and applications of $L^2$ estimates}
%\footnote{Details must be included, including proper definitions of various terms}
We  now iterate the estimates in Proposition \ref{lpdecouplingprop}.
%proving the estimates in Proposition \ref{dyadicq-Sob}.
%
We give the argument for $f\mapsto f*\mu_{k,l}$, $l\le[k/3]$ (see the definitions ahead of the statement of Proposition
\ref{muklestprop}).
We can use the Heisenberg translations to reduce to the case that $f$ is supported in $\{(y,v):
|y|\le 1,  |v|\le 1\}$. Then $f*\mu_{k,l}$ is supported in  $\{(x,u): |x|+|u|\le C\}$ for some fixed constant.

We shall work with a partition of unity in $(x,u)$ space 
$$\sum_{(x^\circ, u^\circ)\in \cZ_l}\chi\ci{x^\circ,u^\circ}=1$$
where 
$\cZ_l$ is a grid of $c2^{-l}$ separated points and the bump functions 
$\chi\ci{x^\circ,u^\circ}$ 
are associated in a natural way with cubes of diameter $O(2^{-l})$ centered at the points in the grid.

Then for $f$ with support in a fixed ball near the origin we get
$$\|f*\mu_{k,l}\|_q \lc \Big(\sum_{(x^\circ, u^\circ)\in \cZ_l}
\big\|\chi\ci{x^\circ,u^\circ}  (f*\mu_{k,l})\big\|_q^q\Big)^{1/q}
$$

We define numbers $\delta_j$ of the form $2^{-m_j}$ with $m_j\in \bbN$ as follows.
Let $m_0=[l\eps/100n ]$ and $\delta_0=2^{-m_0}$.
%Define $m_j=[2m_{j-1}]$ if $m_{j-1}<l/3$. A finite application of this iteration will yield an index $\underline j$ with $m_{\underline j}\ge l/3$. Define for $j\ge \underline  j$, 
 Define for $j\ge 1$
 $$m_j = \Big\lfloor \frac{l+m_{j-1}}2\Big\rfloor,$$ note that $m_j \to l$.
We will stop the process when $m_j > l(1-\frac{\eps}{10n})$. Let $j_*$ be the smallest integer greater than $l(1-\frac{\eps}{10n})$.

Decompose $\bbR^{2n-1}$  into disjoint dyadic cubes of side length $2^{-m_0}$, call these cubes
$Q^0_\nu$. Let $f_{0,\nu}(y,v)= f(y,v)\bbone_{Q^0_\nu}(y,v)$. Then by Minkowski's and H\"older's inequality  
\begin{align}\label{indstart}\|f*\mu_{k,l}\|_q &\lc \sum_{\nu_0} \Big(\sum_{(x^\circ, u^\circ)\in \cZ_{l}}
\big\|\chi\ci{x^\circ,u^\circ} ( f_{0,\nu_0}*\mu_{k,l})\big\|_q^q\Big)^{1/q}
\\
&\lc 2^{m_0(2n+1)/q'} \Big(\sum_{\nu_0} \sum_{(x^\circ, u^\circ)\in\cZ_l}
\big\|\chi\ci{x^\circ,u^\circ} ( f_{0,\nu_0}*\mu_{k,l})\big\|_q^q\Big)^{1/q}. \notag
\end{align}

Fix $j\le j_*$ and let $\{Q^j_{\nu}\}$ be the collection of dyadic cubes of sidelength $2^{-m_j}$.
%Now decompose each cube $Q^0_{\nu_0}$ as a disjoint union of dyadic cubes  of side length $2^{-m_1}$ and get the corresponding decomposition on the functions,
Set $f_{j,\nu}=f\bbone_{Q^j_\nu}$.
We claim, for some $C(\eps_1)$  and any $N_2\in \bbN$  the bound
\begin{align}\label{indclaim} 
\|f*\mu_{k,l}\|_q &\le C_0 C_1(\eps_1)^j 2^{m_0\frac{2n+1}{q'}}  2^{(m_j-m_0)((2n-1)(\frac 12-\frac 1q)+\eps_1)}
\\&\qquad \qquad\times \Big(
\sum_{\nu_j}  
\sum_{(x^\circ, u^\circ)\in \cZ_l}
\big\|\chi\ci{x^\circ,u^\circ} ( f_{j,\nu_j}*\mu_{k,l})\big\|_q^q\Big)^{1/q}\notag
\\&+ j C_2(\eps,N_1) C_1(\eps_1)^{j-1}  2^{(2n+1)l} 2^{-kN_1} \|f\|_q.\notag
\end{align}
For $j=0$ this holds by \eqref{indstart}. Suppose \eqref{indclaim} holds for $j=J<j_*$.
%$$f_{0,\nu_0}= \sum_{\nu_1:Q^1_{\nu_1}\subset Q^0_{\nu_0}} f_{1,\nu_1}.$$
We apply Proposition \ref{lpdecouplingprop} to bound
%\footnote{Should be better explained} 
\begin{align*}
&\|f*\mu_{k,l}\|_q \le  C(\eps_1)^{J+1} 2^{m_0(2n+1)/q'}  2^{(m_{J+1}-m_J+m_J-m_0)((2n-1)(\frac 12-\frac 1q)+\eps_1)}\\ 
&\qquad \qquad\times \Big(\sum_{\nu_J} 
\sum_{\nu_{J+1}: Q^{J+1}_{\nu_{J+1}}\subset Q^J_{\nu_J}}  
\sum_{(x^\circ, u^\circ)\in \cZ_l}
\big\|\chi\ci{x^\circ,u^\circ} ( f_{J+1,\nu_{J+1}}*\mu_{k,l})\big\|_q^q\Big)^{1/q}
\\&\qquad\qquad+ (J+1)  C(\eps,N_1)C(\eps_1)^{J}  2^{(2n+1)l} 2^{-kN_1} \|f\|_q .
\end{align*}
Choosing $N_1>N_2+ 2n+1$ gives
\begin{align*}
&\|f*\mu_{k,l}\|_q\,\le C(\eps_1)^{J+1} 2^{m_0\frac{2n+1}{q'}}  2^{(m_{J+1}-m_0)((2n-1)(\frac 12-\frac 1q)+\eps_1)}
\\&\qquad \qquad\times \Big(
\sum_{\nu}  
\sum_{(x^\circ, u^\circ)\in \cZ_l}
\big\|\chi\ci{x^\circ,u^\circ} ( f_{J+1,\nu}*\mu_{k,l})\big\|_q^q\Big)^{1/q}
\\&\qquad\qquad+ (J+1)C(\eps_1)^J \tilde C(\eps,N_2)  2^{-kN_2} \|f\|_q.
\end{align*}

We apply this for $J= j_*-1$ and observe that $j_* \lc \eps^{-1}+\log l$.
\begin{align*}
\|f*\mu_{k,l}\|_q &\lc C(\eps_1)^{C_2(\log l+\eps^{-1})} 2^{m_0\frac{2n+1}{q'}}  2^{(m_{j_*}-m_0)((2n-1)(\frac 12-\frac 1q)+\eps_1)}
\\&\quad\times \Big(
\sum_{\nu}  
\sum_{(x^\circ, u^\circ)\in \cZ_l}
\big\|\chi\ci{x^\circ,u^\circ} ( f_{j_*,\nu}*\mu_{k,l})\big\|_q^q\Big)^{1/q}
\\&\quad+ C(\eps_1)^{C_2(\log l +\eps^{-1})} C(\eps, N_2)  (\eps^{-1} +\log l) 2^{-kN_2}\|f\|_q
\end{align*}

%d in a dyadic cube of side length $2^{-m_{j_*}}\approx \delta_{j_*}\approx 2^{l(1-\eps)}$.

%We use an initial trivial decoupling argument at scale $\delta_0$ (simply by using Minkowski's argument, resulting in an additional constant $\delta_0^{-2n-1}$.

We use the definition of $m_0$ and $m_{j_*}$, and that
$A^{\log j} \le C(\delta,A) 2^{\delta j}$ for any $\delta>0$. Thus, for $2\le q\le \frac{4n+2}{2n-1}$,
\begin{multline} \label{decappl} 
\|f*\mu_{k,l} \|_q \lc_{\eps,\eps_1} 2^{l\eps/10} 
 2^{l(1-\eps) (2n-1) (\frac 12-\frac 1q+\eps_1)}
\Big( \sum_\nu \|f_{j_*,\nu}*\mu_{k,l} \|_q^q\Big)^{1/q}\\+ C_{N,\eps}2^{-kN}\|f\|_q
\end{multline}
%(or a slightly better estimate with irrelevant improvements).
We use the $L^2$ estimates for Fourier integral operators associated with folding canonical relations, as  in \cite{MuSe} (relying on  the version in \cite{GS}). 
We get for a bounded set $\cU$
\Be \label{L2appl} 
 \Big( \sum_\nu \|f_{j_*,\nu}*\mu_{k,l} \|_{L^2(\cU)}^2\Big)^{1/2}  \lc  2^{-k (2n-1)/2}2^{l/2} 
\|f\|_2
\Ee
for $l\le [k/3]$.
We also have a trivial $L^\infty$ bound, using that the projection of the support of $f_{j_*,\nu}$ to $\bbR^{2n-1}$ is contained in a ball of radius $c 2^{-m_{j_*}}\approx 2^{-l (1-\frac {\eps}{10 n})}$.  We get
\Be \label{L2appl} 
\sup_\nu \|f_{j_*,\nu}*\mu_{k,l} \|_\infty \lc  2^{-l(1-\frac{\eps}{10 n})(2n-1)} \|f\|_\infty.
\Ee
By interpolation 
\Be\label{interpol}
\Big( \sum_\nu \|f_\nu*\mu_{k,l} \|_{L^q(\cU)}^q\Big)^{1/q}  \lc_\eps 2^{-k(2n-1)/q} 2^{-l (2n-1) +l(4n-1)/q} 2^{l\eps/2}\|f\|_q
\Ee
We combine  \eqref{decappl}  and \eqref{interpol} and obtain 
\Be \label{decapplcomb} 
\|f*\mu_{k,l} \|_{L^q(\cU)} \lc_\eps  2^{l\eps} 2^{l(\frac{2n}q-\frac{2n-1}{2})} 2^{-k\frac{2n-1}{q}}\|f\|_q
\Ee
for $l\le [k/3]$ and  (choosing $\eps$ small) we can sum in $l$ if $q>\frac{4n}{2n-1}$.

Equivalently we  obtain
\Be\|f*\mu_k\|_q\le C(q)  2^{-k(2n-1)/q} \|f\|_q \text{ for } \frac{4n}{2n-1} <q\le \frac{4n+2}{2n-1},
\label{restrictedrange}
\Ee 
provided that $f$ is supported in a fixed ball centered at the origin, say in $Q=[0,1)^{2n+1}$.

We now  remove this assumption on the support of $f$ in \eqref{restrictedrange}.
For $m\in \bbZ^{2n+1}$ let $$Q_m=m\cdot Q= \{(\overline m+y, m_{2n+1}+v +\overline m^\intercal Jy: y\in Q\}.$$
Let $\mu_k$ be supported in $R_A=\{(x,u): |x|\le A, |u|\le A \}$ for some $A\ge 1$. Then 
$(f\bbone_Q) *\mu_k$ is supported on the set
$$\{(x,u): |x|\le A+\sqrt{2n}, \,|u|\le A+1+(A+\sqrt{2n}\|J\|\sqrt{2n}\};$$
% where $|x|\le A+\sqrt{2n}$, $|u|\le 1+A\sqrt{2n}\|J\|$
to see this  write $x^\intercal Jy=(x-y)^\intercal Jy$. Thus 
$(f\bbone_Q) *\mu_k$ is supported  in $R_B$ with   $B=B(A)$
(and $B(A)$ is the maximum of the bounds for $|x|$ and $|u|$ in the displayed formula).
%(and $B(A)= \max\{1+\sqrt {2n}, 1+A\sqrt{2n}\|J\|\}$. 
By left translation 
$(f\bbone_{Q_m}) *\mu_k$ is supported in $m\cdot R_B$
and we have 
\Be\|(f\bbone_{Q_m}) *\mu_k\|_q\le C(q)  2^{-k(2n-1)/q} 
\|f\bbone_{Q_m}\|_q 
\label{restrictedrange}
\Ee 
$\text{ for } \tfrac{4n}{2n-1} <q\le \tfrac{4n+2}{2n-1}$,
uniformly in $m$.
Now for each $m\in \bbZ^{2n+1}$ the cardinality 
of $$\{\widetilde m\in \bbZ^{2n+1}: m\cdot R_B\cap \widetilde m\cdot R_B\neq \emptyset\}$$ is bounded above by 
$C(B)^{2n+1}$. Indeed,
% \footnote{is this paragraph needed?}
 let  $m\cdot R_B\cap \widetilde m\cdot R_B\neq \emptyset$ which is equivalent with
$R_B\cap m^{-1} \widetilde m R_B\neq \emptyset.$ Let $(w,t)= m\cdot \widetilde m^{-1}$ and $(x,u)\in R_B$.
Then $(w,t)\cdot (x,u)=(w+x, t+u+w^\intercal J x)$. if $|w|\ge 2B $ then $(w,t)\cdot (x,u)\notin R_B$. 
If $|w|\le 2B$ and $t>2B +2B^2\|J\|$ then 
$|t+u+w^\intercal J x|>B$ and again $(w,t)\cdot (x,u)\notin R_B$. 
Apply this with $$(w,t)=m^{-1}\cdot \tilde m= ( {\overline {\tilde m}}  -\overline m, \tilde m_{2n+1}- m_{2n+1}-\overline m^\intercal J(\overline {\tilde{m}}-\overline m)),$$ and clearly for fixed $m$ there are only $C(B)^{2n+1}$ integer vectors $\tilde m$ with $|m^{-1} \cdot \tilde m| \le 2B+2B^2\|J\|$.

Hence for general $f\in L^q(\bbH^n)$, $\frac{4n}{2n-1} <q\le \frac{4n+2}{2n-1}$, 
\begin{align*} &\|f*\mu_k\|_q= \Big\|\sum_{m\in \bbZ^{2n+1}} (f\bbone_{Q_m})*\mu_k\Big\|_q
\lc_B \Big(\sum_{m\in \bbZ^{2n+1}} \|(f\bbone_{Q_m})*\mu_k\|_q^q\Big)^{1/q}
\\
&\lc_{B,q} 
  2^{-k(2n-1)/q} 
\Big(\sum_{m\in \bbZ^{2n+1}} \|f\bbone_{Q_m}\|_q^q\Big)^{1/q}
\lc_{B,q} 
  2^{-k(2n-1)/q} \|f\|_q.
  \end{align*}

Now  convolution with $\mu_k$ is uniformly bounded on $L^\infty$ and thus interpolation 
gives
\Be \label{f*muk}
\|f*\mu_k\|_q\le C(q)  2^{-k(2n-1)/q} \|f\|_q \text{ for } \frac{4n}{2n-1} <q\le \infty.
\Ee
Alternatively one can argue with 
as  in \cite{pr-se}  with the Wolff version of decoupling for $q> \frac{4n+2}{2n-1}$.
By duality \eqref{f*muk} also implies  
\Be\label{f*mukdual}\|f*\mu_k\|_p\le C(p)  2^{-k(2n-1)(1-\frac 1p)} \|f\|_p \text{ for } 1\le p< \frac{4n}{2n+1}.
\Ee 

By modifying the definition of $\gamma$ in \eqref{Kdefin} we also get
the same estimate for $\mu_k$ replaced with $2^{-k} \frac{d}{ds} \dil_s\mu_{k}$ when $s\approx 1$.
This proves Proposition \ref{dyadicq-Sob}.
A standard Sobolev imbedding argument yields Corollary \ref{maxdyadic}.

\medskip

 \noi{\it Remark.} Up to this point we worked with a measure $\mu$ in $\bbR^{2n}$ supported on a graph  
 $x_{2n}= g(x_1,\dots, x_{2n-1})$, with $D^2g$ nondegenerate. We must also  also consider  the case
 where  the surface is given by $x_j=g(x_1,\dots, x_{j-1}, x_{j+1},\dots)$. However this situation can be reduced to the former by permuting the variables; one just needs to note that the change of variables argument in \S\ref{changesofvar} applies, and that the skew symmetric matrix $J$ in the former case  is replaced by $P^\intercal JP$ after a change of variables, with   $P$   a suitable permutation matrix.

\subsection*{\it $L^p$ Sobolev result}
%\noi {\it Remark.}
Proposition \ref{dyadicq-Sob} can be reformulated as a regularity result in Besov spaces for functions in $\bbR^{2n+1}$ which are supported on a compact set. However one can  combine  Proposition \ref{dyadicq-Sob}, part (i)  with the result in Theorem 1.1 of \cite{prs} to show a better result using Sobolev spaces $L^q_\alpha(\bbR^{d})$
for $\alpha= (d-2)/q$, with $d=2n+1$.

\begin{corollary} Let $\cU$ be a compact neighborhood of the origin  
%in $\cH^n=\bbR^{2n+1}$ 
and let 
$L^q_\alpha(\bbR^{2n+1})$ be the usual Sobolev space.
Let $\cR f=f*\mu$ with $\mu$ as above, and with the convolution on the Heisenberg group $\bbH^n$.
Then we have for $\frac{4n}{2n-1} <q< \infty$, 
$$\| \cR f\|_{L^q_{(2n-1)/q}} \le C(\cU,q)  \|f\|_q$$
whenever $f$ is supported in $\cU$.
\end{corollary}
See the discussion in \S2 of \cite{prs} for related  examples.
Theorem 1.1 of \cite{prs}    actually gives a  better statement
using Besov and Triebel-Lizorkin spaces, namely that  $\cR: (B^0_{q,q})_{\text{comp}}\to F^{(2n-1)/q}_{q,r}$  for all $r>0$, in the $q$-range of the corollary.

  \section{Estimates for the global maximal operator} \label{globalsect}
By the Marcinkiewicz interpolation theorem it suffices to prove a weak type $(p,p)$ estimate for $p> \frac{2n}{2n-1}$, i.e.
\Be \label{wt}\meas\big(\big\{(x,u): \sup_t |f*\dil_t \mu (x,u)|>\alpha \big\}\big) \lc \alpha^{-p}\|f\|_{L^p(\bbH^n)}^p 
\Ee
%with $\eps(p)>0$ 
for $p>\frac{2n}{2n-1}$.
%We shall prove 

We use Calder\'on-Zygmund theory on the Heisenberg group, with respect to the nonisotropic balls
\begin{align*}B((a,b),\delta)&= \{(y,v): |(y,v)^{-1}\cdot (a,b) |\le \delta \} \\
&=\{(y,v): |(a-y, b-v-y^\intercal J a)|\le \delta\}
\end{align*}
see for example \cite{CW},  \cite{FS}, or 
%chapter I, \S4 of 
\cite{Steinbook3}.
We apply the Calder\'on-Zygmund decomposition  to the function $|f|^p$.
%Theorem 2 on page 1, and  in  \cite{Steinbook3}, and its proof,  to the $|f(x)|^p \arg (f(x))$.

Let $$\Om_\alpha\equiv \Omega_\alpha(f)=\{ (x,u): M_{HL}( |f|^p) (x,u)>\alpha^p\}\,$$
so that \Be\label{HLbd} \meas (\Om_\alpha)\lc 
\alpha^{-p} \|f\|_p^p.\Ee
Let $g_1 = f\bbone_{\Om_\alpha^\complement}$. Then $|g_1(x)|\le \alpha$ almost everywhere

%One  decomposes $f=g+b$  where \Be |g(x)|^p \le C\alpha^p \text{ a.e. }.\Ee 
We have $\Omega_\alpha = \cup_{Q\in \cQ_\alpha} Q$
where the sets $Q$ in the family $\cQ_\alpha$  are disjoint and measurable and  there are  constants $c_1$, $c_2$ such that 
$2\le c_1< c_2/8$ such that for every $Q$ there is a point $P_Q$ and an $r_Q>0$ with
$$B(P_Q, r_Q)\subset Q\subset B(P_Q, c_1 r_Q)  \subset B(P_Q, 2c_1 r_Q) \subset \Omega_\alpha;$$
moreover $$\sum_Q \bbone_{B(P_Q, 2c_1 r_Q)}(x)
 \lc C \text{ almost everywhere}$$
 and 
$$B(P_Q, c_2 r_Q)\cap \Omega_\alpha^\complement \neq \emptyset\,.$$
We decompose 
%The function $f\bbone_{\Omega_\alpha} $ can be decomposed as 
$$f\bbone_{\Om_\alpha}= \sum_{Q\in \cQ_\alpha}  f\bbone_Q.$$
Let $\phi$ be a $C^\infty$ function supported in the Euclidean ball of radius 1 centered at the origin and such that $\int \phi=1$. We introduce some cancellation using suitable convolutions as in \cite{carbery}. 
 Let $$\phi_m  (x,u)=2^{-m (2n+2) } \phi(2^{-m}x, 2^{-2m} u)$$ let $m_Q$ be such that 
 $c_1r_Q/2\le 2^{m_Q}< c_1r_Q$ and set
$$ \begin{aligned}g_Q &= (f\bbone_Q) * \phi_{m_Q},
\\
b_Q&= f\bbone_Q - (f\bbone_Q) * \phi_{m_Q}
.
\end{aligned}
$$
 Then $g_Q$, $b_Q$ are  supported in $B(P_Q, 2c_1 r_Q)$. We set $g_2= \sum_{Q\in \cQ_\alpha} g_Q$ and we have $|g_2(x)|\le C\alpha$ almost everywhere. Let $g=g_1+g_2$ and $b=f-g$.
 % Thus 
 Now 
 \begin{align*}
&  \meas\big(\big\{(x,u): \sup_t |f*\dil_t \mu (x,u)|>\alpha \big\}\big) 
\\  &\le  \meas(\Om_\alpha)+ \meas\big(\big\{(x,u): \sup_t |g*\dil_t \mu (x,u)|>\alpha/2 \big\}\big) 
\\&\quad +
\meas\big(\big\{(x,u)\in \Omega_\alpha^\complement : \sup_t |b*\dil_t \mu (x,u)|>\alpha/2 \big\}\big).
  %\lc  \|f\|_{L^p(\bbH^n)}^p 
\end{align*}

Since $n\ge 2$ we know the $L^2$ boundedness of our maximal operator and  the standard argument 
using $\|g\|_\infty\lc\alpha$ gives
\Be\label{wtgood}\begin{aligned}
 &\meas\big(\big\{(x,u): \sup_t |g*\dil_t \mu (x,u)|
 >\alpha/2 \big\}\big) \\&\le
 (\alpha/2)^{-2} \|\sup_t |g*\dil_t \mu |\|_2^2 \\&\le C^2 \alpha^{-2} \|g\|_2^2 \le \widetilde C^p \alpha^{-p}\|f\|_p^p
 \end{aligned}
 \Ee
 
 It suffices to estimate
 % the size of the subset of $\Omega_\alpha$ for which
  $$\meas\big(\big\{((x,u): \sup_t |b*\dil_t \mu (x,u)|>\alpha/2\big\}\big).$$
  
  For $m\in \bbZ$ let $$b^{[m]}= \sum_{\substack{Q\in \cQ_\alpha: \\ m_Q=m}} b_Q\,.$$
  Observe that
  $$\supp(b^{[m]}*\mu_t) \subset \Omega_\alpha  \text{ if } t\le 2^{m-C_1} $$ 
  where $C_1\in \bbN$ and $C_1$ depends only on the support of $\mu$.
  
  If $(x,u)\notin \Omega_\alpha$ we have
  \begin{align}
  &\sup_{t>0} |b*\dil_t\mu(x,u) |
  \le \Big(\sum_{j\in \bbZ} \sup_{1/2<s<1} 
  |b* \dil_{2^js} \mu(x,u)|^p\Big)^{1/p}
  \notag
  \\
  &=
   \Big(\sum_{j\in \bbZ} \sup_{1/2<s<1} 
  \big| \sum_{m\le j+C_1} b^{[m]} * \dil_{2^js} \mu(x,u)\big|^p\Big)^{1/p}
  \notag
  \\
  &\le \sum_{k\ge 0} 
  \Big(\sum_{j\in \bbZ} \sup_{1/2<s<1} 
  \big| \sum_{m\le j+C_1} b^{[m]} * \dil_{2^js} \mu_k(x,u)\big|^p\Big)^{1/p}.
  \label{fixedxstr}
  \end{align}
  
 We have straightforward estimates   
\begin{align*} &\|\mu_k\|_{L^1} + 2^{-k} \|\nabla\mu_k\|_{L^1} \le C,
\\
&2^{-k} 
\Big\|\frac{d}{ds} \dil_s \mu_k\Big\|_{L^1} + 2^{-2k} \Big\|\nabla \frac{d}{ds} \dil_s \mu_k\Big\|_{L^1} \le C, \quad \frac {1}{2}\le s\le 1.
\end{align*} 
which we use for $m\le j- C_2 k$ with sufficiently large  $C_2$, say $C_2=10$.
For this range we estimate the corresponding part in \eqref{fixedxstr}
\begin{align*}
&\Big\|  \sum_{k\ge 0} 
  \Big(\sum_{j\in \bbZ} \sup_{1/2<s<1} 
  \big| \sum_{m\le j-C_2k } b^{[m]} * \dil_{2^js} \mu_k\big|^p\Big)^{1/p}\Big\|_p
  \\
  &\le \sum_{k\ge 0} \sum_{l\ge 0} 
  \Big\|
  \Big(\sum_{j\in \bbZ} \sup_{1/2<s<1} 
  \big|  b^{[j-C_2k -l]} * \dil_{2^js} \mu_k\big|^p\Big)^{1/p}\Big\|_p
  \\
  &= \sum_{k\ge 0} \sum_{l\ge 0} 
  \Big(\sum_{j\in \bbZ}   \big\|  
  \sup_{1/2<s<1} 
  \big|  b^{[j-C_2k -l]} * \dil_{2^js} \mu_k\big| \big\|_p^p\Big)^{1/p}
  \end{align*}
  Now  let $f^{[m]}= \sum_{Q: m_Q=m} f\bbone_Q$ so that $b^{[m]}=f^{[m]}*(\delta-\phi_m)$
  (with $\delta$ the Dirac measure at the origin). Then
    \begin{align*}
 &  \big\|  
  \sup_{1/2<s<1} 
  \big|  b^{[j-C_2k -l]} * \dil_{2^js} \mu_k| \big\|_p
  \\
  &\le \big\|  
      b^{[j-C_2k -l]} * \dil_{2^{j-1}} \mu_k\big\|_p+ \int_{1/2}^1 
   \big\|
  b^{[j-C_2k -l]} * \frac d{ds} \dil_{2^{j}s} \mu_k\big\|_pds
  \\
  &\le \|f^{[j-C_2 k-l]}\|_p  
  \Big[ \|(\delta-\phi_{j-C_2k-l})* \dil_{2^{j-1}}\mu_k\|_{L^1} 
  \\&\qquad\qquad\qquad\qquad\qquad +
   \int_{1/2}^1 \|(\delta-\phi_{j-C_2k-l})* \frac{d}{ds} \dil_{2^{j}s}\mu_k\|_{L^1}  ds \Big]
  \\
  &\lc 2^{2k}  2^{-C_2 k-l}  \|f^{[j-C_2 k-l]}\|_p  
  \end{align*}
  and thus
  \begin{align*}
  &\Big\|  \sum_{k\ge 0} 
  \Big(\sum_{j\in \bbZ} \sup_{1/2<s<1} 
  \big| \sum_{m\le j-C_2k } b^{[m]} * \dil_{2^js} \mu_k\big|^p\Big)^{1/p}\Big\|_p
    \\
  &\lc    \sum_{k\ge 0} \sum_{l\ge 0}
  2^{2k}  2^{-C_2 k-l}  
  \Big(\sum_{j\in \bbZ}  \|f^{[j-C_2 k-l]}\|_p ^p \Big)^{1/p} \\&
  \lc \|f\|_p  \sum_{k\ge 0} 
  2^{(2-C_2)k}  \sum_{l\ge 0} 2^{-l}  \lc\|f\|_p
    \end{align*}
   Note that this estimate holds for all $p\ge 1$.
   
   The main contributions come  from the range $j-C_2 k\le m\le j+C_1$. Here we use Corollary \ref{maxdyadic}  and bound, for fixed $k$,
   \begin{align*}
 &  \Big\|  
  \Big(\sum_{j\in \bbZ} \sup_{1/2<s<1} 
  \big| \sum_{j-C_2 k <m\le j+C_1 } b^{[m]} * \dil_{2^js} \mu_k\big|^p\Big)^{1/p}\Big\|_p
 \\ &\le 
\Big(  \sum_{j\in \bbZ}\sum _{j-C_2 k <m\le j+C_1}  (C_1+C_2 k)^{p/p'}
\| b^{[m]} * \dil_{2^js} \mu_k\|_p^p\Big)^{1/p}
   \\
   &\lc 2^{k(\frac{2n} p-2n+1)}
   \Big(  \sum_{j\in \bbZ}\sum _{j-C_2 k <m\le j+C_1}  (C_1+C_2 k)^{p/p'}
\| b^{[m]} \|_p^p\Big)^{1/p}\\
 &\lc_{C_1, C_2}  (1+k) 2^{k(\frac{2n} p-2n+1)}
   \Big(  \sum_{m\in \bbZ}
\| b^{[m]} \|_p^p\Big)^{1/p} \\&\lc
 (1+k) 2^{k(\frac{2n} p-2n+1)}
   \|f\|_p.
   \end{align*}
   Thus for
    $p>\frac{2n}{2n-1}$,
   $$
   \Big\|  \sum_{k\ge 0} 
  \Big(\sum_{j\in \bbZ} \sup_{1/2<s<1} 
  \big| \sum_{\substack{ j-C_2 k <\\m\le j+C_1} } b^{[m]} * \dil_{2^js} \mu_k\big|^p\Big)^{1/p}\Big\|_p
  \le C_p \|f\|_p \,.
  $$
  We combine the $L^p$ estimates and get 
  $p>\frac{2n}{2n-1}$
  $$
  \big\|\sup_{t>0} |b*\dil_t\mu | \big\|_{L^p(\Omega_\alpha^\complement)}  \lc_p \|f\|_p.$$
  
Hence by 
 Tshebyshev's inequality,
  \Be\label{wtoff}
  \meas\big(\big\{(x,u)\in \Omega_\alpha^\complement : \sup_t |b*\dil_t \mu (x,u)|>\alpha/2 \big\}\big) \lc \alpha^{-p}\|f\|_p^p.
  \Ee
  The desired weak type bound \eqref{wt} follows from
  \eqref{HLbd}, \eqref{wtgood} and  \eqref{wtoff}. \qed

\appendix\section{\it   An  integration by parts lemma}\label{interlude}
To perform the decoupling step we used  a familar integration by parts  lemma (which has been used many times but  is often not  found  in the precise 
form needed in an application). 
We formulate what we need here  and include some  details of the proof, for convenience.
%We include the details here.

Let $h\in C^\infty_c $ function on $\bbR^n$ and  let $w\mapsto \psi(w)$ be a real valued $C^\infty$ function
such that $\nabla\psi\neq 0$ on the support of $h$.

We define $$L h= \text{div} \big( \frac{h \nabla \psi }{|\nabla\psi|^2}\big).$$
Then the formal adjoint 
%$L$ is the formal adjoint of 
$ L^*=- \inn{\frac {\nabla\psi}{|\nabla\psi|^2}}{\nabla} $  satisfies
$i\lambda ^{-1}L^* e^{i\lambda \psi}=e^{i\lambda \psi}$.

We let $L^0 h=h$ and define inductively $L^N h=L L^{N-1} h$.
We then have by integration by parts 
$$\int e^{i\la\psi(w)} h(w) dw = 
\Big(\frac{i}{\la}\Big)^N \int e^{i\la\psi(w)} L^Nh(w) dw 
$$ for $N=1,2,\dots$.

We need to analyze the behavior of $L^Nh$ and the following terminology will be helpful.

\begin{definition} \label{typeABterms}
{\rm 
(i) The term $h$ is of type $(A,0)$. A term is of type $(A,j)$ if it is $h_j/|\nabla\psi|^j$ where $h_j$ is a derivative of order $j$ of $h$. 
%A term is of type $(A)$ if it is of type $(A,j)$ for some $j\ge 0$.

(ii) A term is of type $(B,0)$ if it is equal to $1$. A term is of type $(B,j)$ for some $j\ge 1$  if it  is of the form $\psi_{j+1}/|\nabla\psi|^{j+1}$ where $\psi_{j+1}$ is a derivative of order $j+1$ of $\psi$. }
\end{definition}

\begin{lemma}\label{intbyparts}
Let $N=0,1,2,\dots$. Then $$L_N h = \sum_{\nu=1}^{K(N,n)} c_{N,\nu} h_{N,\nu}$$ where each $h_{N,\nu}$ is of the form
%The term $L^N h$ is a linear combination  of terms of the form
$$P(\tfrac{\nabla \psi}{|\nabla\psi|}) \beta_A \prod_{l=1}^M \gamma_l$$
where $P$ is a polynomial of $n$ variables (independent of $h$ and $\psi$), $\beta_A$ is of type $(A,j_A) $ for some $j_A\in \{0,\dots, N\}$ and the terms $\gamma_l$ are of type $(B,\ka_l)$, so that $j_A+\sum_{l=1}^M \ka_l=N$. The terms $P, \beta_A, \gamma_l$ depend on $\nu$.
\end{lemma}

\begin{proof}[Sketch of proof]The statement holds when $N=0$.
We compute
\begin{align*}La &=  \frac{1}{|\nabla \psi|^2} \sum_{j=1}^n  \frac {\partial \psi}{\partial w_j} \frac{\partial a}{\partial w_j}
+ \frac{a}{|\nabla\psi|^2} \sum_{j=1}^n  \frac{\partial^2 \psi}{(\partial w_j)^2}
\\&- \frac{2a}{|\nabla\psi|^4} \sum_{j=1}^n \frac{\partial \psi}{\partial w_j} \sum_{k=1}^n 
\frac{\partial \psi}{\partial w_k}\frac{\partial^2\psi }{\partial w_j\partial w_k}.
\end{align*}
In particular one can check immediately that the assertion holds for $N=1$.

Now let $a$ be of type $(A,m)$, %i.e. $a=h_m/|\nabla\psi|^m$ 
and $\ga$ be of type $(B,m)$.
%i.e.  $\ga=\psi_{m+1}/|\nabla\psi|^{m+1}$ where $\psi_{m+1}$ is a derivative  of $\psi$, of order $m+1$.
One observes 
that
\begin{align*}
\frac{1}{|\nabla \psi|^2} \frac{\partial\psi }{\partial w_j}\frac{\partial a} {\partial w_j} 
&= P(\tfrac{\nabla\psi}{|\nabla\psi |}) a_{m+1} 
+a\sum_{i=1}^n
P_{A,i}(\tfrac{\nabla\psi}{|\nabla\psi |}) b_{1,i}
\\
\frac{1}{|\nabla \psi|^2} \frac{\partial\psi }{\partial w_j}\frac{\partial \gamma} {\partial w_j} 
&= P(\tfrac{\nabla\psi}{|\nabla\psi |}) b_{m+1} 
+\gamma\sum_{i=1}^n
P_{B,i}(\tfrac{\nabla\psi}{|\nabla\psi |})
 b_{1,i}
\end{align*}
where 
%$P_1$, $P_2$ are polynomials of degree $1$, $2$ respectively, 
$a_{m+1}$ is of type $(A, m+1)$,
$b_{m+1}$ is of type $(B, m+1)$, and the 
$b_{1,i}$ are of type $(B,1)$.
% and  $b_m$ is of type $(B,m)$.
The polynomials are given by $P(x)=x_j$ and $P_{A,i}(x)=-mx_ix_j$,
 $P_{B,i}(x)=-(m+1)x_ix_j$.

Next, if $P$ is a polynomial of $n$ variables then
$$\frac{1}{|\nabla \psi|^2} \frac{\partial\psi }{\partial w_j}\frac{\partial } {\partial w_j} 
(P(\tfrac{\nabla\psi}{|\nabla\psi|})) =\sum_{i=1}^n\tilde  P_i (\tfrac{\nabla\psi}{|\nabla\psi|}) b_i
$$ where the $b_i$ are  of type $(B,1)$ and $\tilde P_i(x)=x_j(\frac{\partial P}{\partial x_i} -x_i \sum_{k=1}^n x_k \frac{\partial P}{\partial x_k})$.

Concerning the other terms in the definition of $Lb$ (with $b$ equal to $a$ or $\gamma$) we note that  $b \frac{\partial^2 \psi}{(\partial w_j)^2}|\nabla\psi|^{-2}$ is a product of $b$ and  a type $(B,1)$ term and
\[\frac{b}{|\nabla\psi|^4} \frac{\partial \psi}{\partial w_j} 
\frac{\partial \psi}{\partial w_k}\frac{\partial^2\psi }{\partial w_j\partial w_k}\] is a product of 
$P_2(\frac{\nabla\psi}{|\nabla\psi |}) b$ and  a type $(B,1)$ term, with $P_2(x)=x_jx_k$.

One combines these observations  and uses them together with the Leibniz rule to carry out the induction step.
\end{proof}

% \section{Bounds for the restricted maximal function}


\begin{thebibliography}{99}

%\bibitem{bghs} D. Beltran, S. Guo, J. Hickman, A. Seeger, 
%\emph{The circular maximal function on the Heisenberg group.} 
%in  preparation.

\bibitem{bourgain-demeter} J. Bourgain,  C.  Demeter,  \emph{The proof of the $\ell^2$  decoupling conjecture,} Ann. of Math. (2) 182 (2015), no. 1, 351--389.

\bibitem{carbery} A. Carbery, \emph{Variants of the Calder\'on-Zygmund theory for $L^p$-spaces.} Rev. Mat. Iberoamericana 2 (1986), no. 4, 381--396.

\bibitem{CW} R.R.  Coifman, G.  Weiss,  Analyse harmonique non-commutative sur certains espaces homog\`enes. (French) \'Etude de certaines int\'egrales singuli\`eres. Lecture Notes in Mathematics, Vol. 242. Springer-Verlag, Berlin-New York, 1971. v+160 pp.

\bibitem{FS} G. Folland, E.M. Stein, Hardy spaces on homogeneous groups. Princeton Univ. Press, Princeton University 1982.

\bibitem{GS}  A. Greenleaf,  A. Seeger, \emph{
Fourier integral operators with  fold singularities,}
  J. reine ang. Math., {455}, (1994),
 35--56.



\bibitem{gs-studia}
\bysame,  \emph{On oscillatory integral operators with folding canonical relations,}  Studia Math.  {132}  (1999),  no. 2, 125--139.

\bibitem{hoer}  L. H\"ormander, \emph{Fourier integral operators.} I. Acta Math. 127 (1971), no. 1-2, 79--183. 

\bibitem{LaWo} I. \L aba,  T.  Wolff, \emph{A local smoothing estimate in higher dimensions.} 
Dedicated to the memory of Tom Wolff. 
J. Anal. Math. 88 (2002), 149--171. 

\bibitem{MuSe} D. M\"uller,  A.  Seeger, \emph{Singular spherical maximal operators on a class of two step nilpotent Lie groups,}  Israel J. Math.  141  (2004), 315--340. 

\bibitem{NaTh} E.K. Narayanan, S.  Thangavelu, \emph{
An optimal theorem for the spherical maximal operator on the Heisenberg group.}
Israel J. Math. 144 (2004), 211--219. 

\bibitem{NevoTh} A. Nevo,  S. Thangavelu, \emph{Pointwise ergodic theorems for radial averages on the Heisenberg group,}  Adv.  Math. 127 (1997), 307--334.

\bibitem{prs} M. Pramanik, K.M. Rogers, A.  Seeger, \emph{A Calder\'on-Zygmund estimate with applications to generalized Radon transforms and Fourier integral operators.} Studia Math. 202 (2011), no. 1, 1-15. 

\bibitem{pr-se} M. Pramanik,  A. Seeger, 
\emph{$L\sp p$ regularity of averages over curves and bounds for
  associated maximal operators,}  Amer. J. Math.  129  (2007),  no. 1,
61--103.


%\bibitem{Pramanik-Seeger} \bysame,
%\emph{$L\sp p$ Sobolev regularity of a restricted X-ray transform in $\Bbb R\sp 3$.}  Harmonic analysis and its applications,  47--64, Yokohama Publ., Yokohama, 2006. 

\bibitem{pr-se-var} \bysame, \emph{Optimal $L^p$-Sobolev regularity of a class of generalized Radon transforms.}
Manuscript.

%\bibitem{riviere} N.M. Rivi\`ere.

\bibitem{PS} D.H. Phong and E.M. Stein,
\emph{Hilbert integrals, singular integrals, and Radon transforms.} I. Acta Math. 157 (1986), no. 1-2, 99--157. 

\bibitem{SS} C.D. Sogge and E.M. Stein,
\emph{Averages over hypersurfaces. Smoothness of generalized Radon transforms.} J. Analyse Math. 54 (1990), 165--188. 

\bibitem{Steinbook3}
E.M.~Stein, Harmonic Analysis: real-variable methods,
                  orthogonality, and oscillatory integrals.
With the assistance of Timothy S. Murphy.
Princeton University Press, Princeton, NJ, 1993.

%\bibitem{TV} T. Tao and A. Vargas, \emph
%{A bilinear approach to cone multipliers},
%\emph{I. Restriction estimates},
% Geom. Funct. Anal. \textbf{10} (2000),  185--215;\emph{
%II, Applications},  Geom. Funct. Anal.  \textbf{10}  (2000),   216--258.

%\bibitem{tataru} D. Tataru,
%\emph{On the regularity of boundary traces for the wave equation},
%Ann. Scuola Norm. Sup. Pisa Cl. Sci. (4) 26 (1998), no. 1, 185--206.


%\bibitem{WolffAnnals}
%T. Wolff, \emph{A sharp bilinear cone restriction estimate,}
% Ann. of Math.,  \textbf{153}, (2001), 661--698.



\bibitem{Wolff1}
T.~Wolff, \emph{Local smoothing type estimates on ${L}^p$ for large $p$,} Geom.
  Funct. Anal. {10} (2000), no.~5, 1237--1288.


\end{thebibliography}
\end{document}